\documentclass[11pt]{amsart}
\usepackage{amssymb,latexsym,amsmath,graphicx,graphics,epic,eepic,enumerate}

\theoremstyle{plain}
\newtheorem{theorem}{Theorem}
\numberwithin{theorem}{section}
\newtheorem{lemma}[theorem]{Lemma}
\newtheorem{proposition}[theorem]{Proposition}

\theoremstyle{definition}
\newtheorem{definition}[theorem]{Definition}
\newtheorem{defn}[theorem]{Definition}

\newtheorem{remark}[theorem]{Remark}
\newtheorem{rmk}[theorem]{Remark}

\newcommand{\C}{{\mathbb C}}
\newcommand{\R}{{\mathbb R}}
\newcommand{\Z}{{\mathbb Z}}
\newcommand{\Q}{{\mathbb Q}}
\renewcommand{\P}{{\mathbb P}}
\newcommand{\s}{{\mathbb S}}

\def\C{\mathbb{C}}

\def\Q{\mathbb{Q}}
\def\R{\mathbb{R}}
\def\Z{\mathbb{Z}}
\def\CP{\mathbb{CP}}

\def\tl{\triangleleft}
\def\w{\omega}

              \newcommand{\E}{{\mathcal E}}

\begin{document}
\title{Symplectic rational $G$-surfaces and equivariant symplectic cones}
\author{Weimin Chen, Tian-Jun Li, and Weiwei Wu}

\AtEndDocument{\bigskip{\footnotesize
  \textsc{Weimin Chen, Department of Mathematics, Unviersity of Massachusetts at Amherst, Amherst, MA 01003} \par
  \textit{E-mail address}: \texttt{wchen@math.umass.edu} \par

}\bigskip{\footnotesize
  \textsc{Tian-Jun Li, School of Mathematics,
      University of Minnesota,
      Minneapolis, MN 55455} \par
  \textit{E-mail address}: \texttt{tjli@math.umn.edu} \par

}\bigskip{\footnotesize
  \textsc{Weiwei Wu, Department of Mathematics,
      University of Georgia,
      Athens, GA 30606} \par
  \textit{E-mail address}: \texttt{weiwei@math.uga.edu} \par

}}

\subjclass[2000]{}
\thanks{}
\date{\today}
\maketitle

\begin{abstract}
We give characterizations of a finite group $G$ acting symplectically on a rational surface 
($\C\P^2$ blown up at two or more points).  In particular, we obtain a symplectic 
version of the dichotomy of $G$-conic bundles versus $G$-del Pezzo surfaces for the corresponding 
$G$-rational surfaces, analogous to a classical result in algebraic geometry. Besides the characterizations 
of the group $G$ (which is completely determined for the case of $\CP^2\# N\overline{\CP^2}$, 
$N=2,3,4$), we also investigate the equivariant symplectic minimality and equivariant symplectic 
cone of a given $G$-rational surface.
\end{abstract}


\section{Introduction}

In this paper we study symplectic $4$-manifolds $(X,\omega)$ equipped with a finite
symplectomorphism group $G$, where $X$ is diffeomorphic to a rational surface.  
We shall call such a pair, i.e., $((X,\omega),G)$, a {\bf symplectic rational $G$-surface}.

They are the symplectic analog of \textbf{(complex) rational $G$-surfaces} studied in algebraic geometry, which are rational surfaces equipped with a holomorphic $G$-action.  These rational 
$G$-surfaces played a central role in the classification of finite subgroups of the plane 
Cremona group, a problem dating back to the early 1880s, see \cite{DI}. 

Note that any rational $G$-surface can be regarded
as a symplectic rational $G$-surface -- simply endowing it with a $G$-invariant K\"{a}hler form
which always exists. Our work shows that a large part of the story regarding the classification of
rational $G$-surfaces can be recovered by techniques from $4$-manifold theory and symplectic topology. Furthermore, we also add some new, interesting symplectic geometry aspect to the
study of rational $G$-surfaces; in particular, in regard to the equivariant symplectic minimality and 
equivariant symplectic cone of the underlying smooth action of a rational $G$-surface.
In addition, we also obtain some result which does not seem previously known in the algebraic geometry literature (cf. Theorem \ref{t:sixPoints}).

We begin with a discussion on the notion of minimality (i.e., $G$-minimality) in the equivariant context. 
Let $(X,\omega)$ be a symplectic $4$-manifold with a finite symplectomorphism group $G$. Suppose there
exists a $G$-invariant set of disjoint union of symplectic $(-1)$-spheres in $X$. Then blowing
down $X$ along the $(-1)$-spheres gives rise to a symplectic $4$-manifold 
$(X^\prime,\omega^\prime)$, which can be arranged so that $G$ is natually isomorphic to
a finite symplectomorphism group of $(X^\prime,\omega^\prime)$. The symplectic $G$-manifold
$X$ is called {\it minimal} if no such set of $(-1)$-spheres exists. It was shown in \cite{C4} that when 
$X$ is neither rational nor ruled, the symplectic $G$-manifold is minimal if and only if the underlying 
smooth manifold is minimal. However, in the case considered in the present paper, the underlying 
rational surface is often not minimal even though the corresponding symplectic rational $G$-surface 
is minimal. Furthermore, it is not known whether the notion of $G$-minimality is the same for the various
different categories, i.e., the holomorphic, symplectic, or smooth categories. In general it is a difficult problem
to establish the equivalence of $G$-minimality in the different categories, and we refer the reader to \cite{C4}
for a more thorough discussion on this topic. For our purpose in this paper, it suffices to only study 
minimal symplectic rational $G$-surfaces.

The most fundamental problem in our study is to classify symplectic rational $G$-surfaces up to equivariant 
symplectomorphisms. However, work in \cite{C1,C2,C3,C4} showed that even in the simple case where $X$ 
is $\C\P^2$ or a Hirzebruch surface and $G$ is a cyclic or meta-cyclic group, such a classification is already 
quite involved. In fact, in one circumstance where $G$ is meta-cyclic,  a weaker classification, i.e., classification up to equivariant diffeomorphisms, still remains open. 

With the preceding understood, the main objectives of this paper are more basic: for symplectic rational $G$-surfaces $X$ in general, we would like to

\begin{enumerate}[(1)]
  \item classify the possible symplectic structures;
  \item describe the induced action of $G$ on $H_2(X)$;
  \item give a list of possible finite groups for $G$;
  \item understand the equivariant minimality and equivariant symplectic cones.
\end{enumerate}

These problems, however, are still highly non-trivial and not completely settled.  In particular, 
part of our determination of $G$ and the induced action on $H_2(X)$ relies on the 
Dolgachev-Iskovskikh's solution of the corresponding problems in algebraic geometry, with new inputs from Gromov-Witten theory and a detailed analysis of the symplectic structures.

\subsection{The setup}
In this paper, we shall be focusing on the case where the rational surface, denoted by $X$, is 
$\C\P^2$ blown up at $2$ or more points. More concretely, we shall consider minimal symplectic rational $G$-surfaces $(X,\omega)$ where $X=\C\P^2\# N\overline{\C\P^2}$, for $N\geq 2$. (Note that the minimality assumption implies in particular that $G$ is a nontrivial group.) The case where the rational surface is $\C\P^2$ or a Hirzebruch surface had been previously studied, cf. \cite{C1,C2,C3,C4}; we point out that the $J$-holomorphic curve 
techniques employed in this paper are drastically different in flavor from those developed in these previous works.

For convenience, we shall fix some notations and terminology, which will be frequently used throughout the paper.
We will denote by $H, E_1,E_2,\cdots, E_N$ a basis of $H^2(X,\Z)$, under which the intersection matrix takes its
standard form, i.e., $H^2=1$, $E_i^2=-1$, $H\cdot E_i=0$, $\forall i$, and $E_i\cdot E_j=0, \forall i\neq j$.
The canonical class of $(X,\omega)$ will be denoted by $K_\omega\in H^2(X)$, in order to emphasize its
dependence on the symplectic form $\omega$. Another frequently used notation is $H^2(X)^G$, which denotes the
subset of $H^2(X,\Z)$ consisting of elements fixed under the induced action of $G$, and is called the {\it invariant lattice}. 

Recall that a symplectic rational surface $(X,\omega)$ is called {\it monotone} if $K_\omega=\lambda [\omega]$ is satisfied 
in $H^2(X;\R)$ for some $\lambda\in\R$. In this case, we have $\lambda<0$, and $N$ must be in the range $N\leq 8$. Such
a symplectic rational surface is the symplectic analog of Del Pezzo surface in algebraic geometry. Another important notion,
given in the following definition and called a {\it symplectic $G$-conic bundle}, corresponds to a conic bundle structure 
on a rational $G$-surface in algebraic geometry.

\begin{definition}
Let $(X,\omega)$ be a symplectic $4$-manifold equipped with a finite symplectomorphism 
group $G$. A symplectic $G$-conic bundle structure on $(X,\omega)$ is a genus-$0$ smooth Lefschetz fibration 
$\pi: X\rightarrow B$ which obeys
the following conditions:
\begin{itemize}
\item each singular fiber of $\pi$ contains exactly one critical point;
\item there exists a $G$-invariant, $\omega$-compatible almost complex structure $J$ such that the fibers of $\pi$ are $J$-holomorphic;
\item the group action of $G$ preserves the Lefschetz fibration.
\end{itemize}

\begin{rmk}\label{rem:}
    Although the above definition looks more rigid than it should be (in particular, it is always an almost complex fibration), Theorem \ref{t:X-structure} shows that this is a purely symplectic notion in the case of minimal symplectic rational $G$-surfaces.
\end{rmk}

A symplectic $G$-conic bundle is called {\it minimal} if for any singular fiber there is an element of
$G$ whose action switches the two components of the singular fiber.
\end{definition}

Here are some immediate consequences from the definition:
\begin{itemize}
\item $X$ is a rational surface if and only if $B=\s^2$; in this case, note that the number of singular fibers of $\pi$ equals
$N-1$, where $X=\C\P^2\# N\overline{\C\P^2}$;
\item the Lefschetz fibration is symplectic with respect to $\omega$;
\item the fiber class lies in the invariant lattice $H^2(X)^G$ as $G$ preserves the Lefschetz fibration;
\item if the underlying symplectic $G$-manifold is minimal, then the symplectic $G$-conic bundle must be also minimal.
\end{itemize}

\vspace{3mm}


\begin{figure}[]
  \centering
  \includegraphics[]{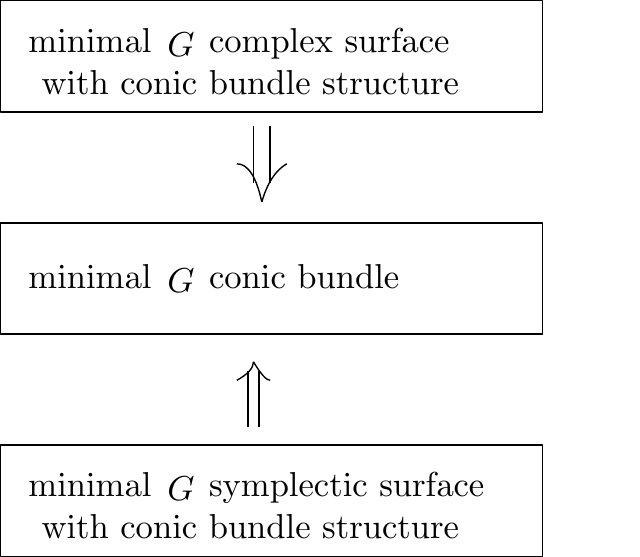}
  \label{fig:figure1}
\end{figure}

\subsection{The symplectic structures}

Our first theorem is concerned with the symplectic structure of a minimal symplectic rational $G$-surface.

\begin{theorem}\label{t:X-structure}
Let $(X,\omega)$ be a minimal symplectic rational $G$-surface, where 
$X=\C\P^2\# N\overline{\C\P^2}$ for some $N\geq 2$.
Then $N\neq 2$, and one of the following holds true:
\begin{itemize}
\item [{(1)}] The invariant lattice $H^2(X)^G$ has rank $1$. In this case, $3\leq N\leq 8$ and 
$(X,\omega)$ must
be monotone.
\item [{(2)}] The invariant lattice $H^2(X)^G$ has rank $2$.  In this case,  $N=5$  or $N\geq 7$, and   there exists a 
symplectic $G$-conic bundle structure on $((X,\omega),G)$.
\end{itemize}
\end{theorem}

\begin{remark}\label{rem:ImplicationOnAG}
\begin{itemize}
\item [{(a)}] An analog of Theorem \ref{t:X-structure}  for minimal rational $G$-surfaces is a classical theorem
in algebraic geometry (see e.g. Theorem 3.8 in \cite{DI}), a proof of which can be given using the 
equivariant Mori theory (see \S 4 of \cite{FH}). With this understood, we remark that our proof of
Theorem \ref{t:X-structure} gives an independent proof of the corresponding result in algebraic
geometry by taking $\omega$ to be a $G$-invariant K\"{a}hler form. Consequently, a significant
portion of the theory of rational $G$-surfaces (e.g. as described in \cite{DI}) can be recovered 
(see Theorems \ref{t:G-structure} \ref{t:Q-structure}).

\item [{(b)}] In case (1) of Theorem \ref{t:X-structure} the invariant lattice $H^2(X)^G$ is spanned by $K_\omega$, and in case (2), $H^2(X)^G$ is spanned by $K_\omega$ and the fiber class of the symplectic $G$-conic bundle. We should point out that the analysis on the type of the group $G$,
the induced action on $H^2(X)$, as well as the structure of the equivariant symplectic cone, 
depends on the rank of the invariant lattice $H^2(X)^G$.

\item [{(c)}] In case (2), the proof of Theorem \ref{t:X-structure} reveals the following additional information 
about the symplectic $G$-conic bundle structure: there exists a basis 
$H, E_1,E_2,\cdots, E_N$ of $H^2(X)$ with standard intersection matrix such that 
\begin{itemize}
\item  [{(i)}] the fiber class is given by $H-E_1$ and the pair of $(-1)$-spheres in a singular fiber 
are given by the classes $E_j$ and $H-E_1-E_j$, where $j=2,\cdots,N$;
\item  [{(ii)}] the symplectic areas satisfy $\omega(E_j)=\frac{1}{2}\omega(H-E_1)$, 
$\omega(E_1)\geq \omega(E_j)$ for $j=2,\cdots, N$;
\item [{(iii)}] the canonical class $K_\omega=-3H + E_1+E_2+\cdots +E_N$.
\end{itemize}

\item [{(d)}]  In complex geometry, it was known that a minimal (complex) rational $G$-surface which is diffeomorphic to $\C\P^2$ blown up at $6$ points must be Del Pezzo  (cf. \cite{DI}, Theorem 3.8, Proposition 5.2). However, it seemed new that the invariant Picard group $\text{Pic}(X)^G$ must be 
of rank $1$.

\end{itemize}    
\end{remark}

\noindent{\bf Comparing the three minimality assumptions:}

\vspace{3mm}

The reader should be noted that a minimal symplectic $G$-conic bundle is different from \textit{a minimal symplectic $G$-surface with $G$-conic bundle structure}: a minimal 
symplectic $G$-conic bundle may still contain a $G$-invariant disjoint union of symplectic 
$(-1)$-spheres.  However, Lemma \ref{l:4-1} implies this cannot happen when $N\ge5$ and $N\neq6$.

The minimal $G$-conic bundles is an intermediate notion between minimal symplectic $G$-surfaces and minimal complex $G$-surfaces.  There is always a $G$-invariant symplectic form compatible with the given complex structure on a minimal complex $G$-surface.  Although we do not know whether this symplectic form is always $G$-minimal, if we assume a symplectic $G$-conic bundle structure underlies this action, this conic bundle structure is always minimal.  Therefore, proving 
our results in the more general minimal $G$-conic bundle context plays an important role in 
bridging the complex and symplectic $G$-surfaces as well as in the study of the equivariant symplectic cones.

\subsection{The homological action and the groups}

Our next task is to describe possible candidates for $G$. 

We begin with case (1) in Theorem \ref{t:X-structure} where $(X,\omega)$ is a minimal symplectic rational $G$-surface such that the invariant lattice $H^2(X)^G$ has rank $1$. 
In this case $(X,\omega)$ is monotone, $H^2(X)^G$ is spanned by $K_\omega$, and 
$3\leq N\leq 8$. With this understood, we note that the orthogonal complement of $K_\omega$ 
in $H^2(X)$ (with respect to the intersection product), denoted by $R_N$, is a $G$-invariant root lattice of type $E_N$ ($N=6,7,8$), $D_5$ ($N=5$), $A_4$($N=4$), and $A_2+A_1$ ($N=3$) respectively. We denote by $W_N$ the corresponding Weyl group.

\begin{theorem}\label{t:G-structure}
Let $(X,\omega)$ be a minimal symplectic rational $G$-surface such that $H^2(X)^G$
has rank $1$. There are two cases:
\begin{itemize}
\item [{(1)}] Suppose $4\leq N\leq 8$. Then the induced action of $G$ on $H^2(X)$ is faithful, which gives rise to a monomorphism 
$\rho: G\rightarrow W_N$. Moreover, the image $\rho(G)$ in $W_N$ satisfies 
$$
\sum_{g\in G} trace\{ \rho(g):R_N\rightarrow R_N\}=0.
$$
\item [{(2)}] Suppose $N=3$. Let $\Gamma$ be the subgroup of $G$ which acts trivially on
$H^2(X)$ and let $K:=G/\Gamma$ be the quotient group. Then $\Gamma$ is isomorphic to
a subgroup of the $2$-dimensional torus and $K$ is isomorphic to $\Z_6$ or the dihedral 
group $D_{12}$. Furthermore, $G$ is a semi-direct product of $\Gamma$ by $K$. As a corollary,
$G$ can be written as a semi-direct product of an imprimitive finite subgroup of $PGL(3)$ by
$\Z_2$.
\end{itemize}
\end{theorem}

\begin{remark}\label{rem:compareG-structure}
       1. For $N=4,5$, the subgroups of the Weyl group $W_N$ which satisfy
the condition $\sum_{g\in G} trace\{ \rho(g):R_N\rightarrow R_N\}=0$ are determined,
see Theorem 6.4 and Theorem 6.9 in \cite{DI} respectively. All such groups can be realized
by a minimal $G$-Del Pezzo surface, which is also minimal as a symplectic rational $G$-surface
with respect to any $G$-invariant K\"{a}hler form (cf. Theorem 1.10(2)). 

2. For $N=3$, Theorem \ref{t:X-structure} (2) and Theorem \ref{t:G-structure} completely determined all possible $G$ that acts minimally on $X$.  The statement in Theorem \ref{t:G-structure} implies the corresponding statement in \cite{DI}, Theorem 6.3. In fact, the semi-direct product structure of $G$ in our statement is an improvement upon the corresponding theorem in algebraic geometry.
For a list of imprimitive finite subgroups of $PGL(3)$ up to conjugacy, see Theorem 4.7 in \cite{DI}.
 \end{remark}

Now we consider case (2) of Theorem \ref{t:X-structure}. In fact, we will work in a slightly more general situation, where the symplectic rational $G$-surface $(X,\omega)$ is not assumed to be minimal, but only admits a minimal symplectic $G$-conic bundle $\pi:(X,\omega)\to\s^2$.  Furthermore, we assume $N\geq 4$ (instead of the fact that $N\geq 5$ when $(X,\omega)$ is a
minimal symplectic rational $G$-surface). 

We make the following definition.

\begin{definition}\label{d:QandG0}\hfill
\begin{itemize}
  \item $Q\tl G$ is the subgroup that acts trivially on the base $\s^2$ of the $G$-conic bundle;
  \item $G_0\triangleleft Q$ is the subgroup that acts trivially on $H^2(X)$ (in the case we consider, there are $N-1\ge 3$ critical values on the base that was fixed, hence $G_0\triangleleft Q$);
  \item $P=G/Q$, so that $G$ decomposes as
  \begin{equation}\label{e:DecompG}
     1\to Q\to G\to P\to 1.
\end{equation}
\end{itemize}
 \end{definition}

 We denote by $\Sigma$ the subset of $\s^2$ which parametrizes the
singular fibers of the symplectic $G$-conic bundle $\pi$. Note that $\# \Sigma=N-1$, 
and the induced action of $P$ on 
$\s^2$ leaves the subset $\Sigma$ invariant. The action of $P$ on $\s^2$ is effective,
so $P$ is isomorphic to a polyhedral group, i.e., a finite subgroup of $SO(3)$.

Therefore, up to an extension problem, the description of $G$ boils down to the following theorem 
which describes the subgroups $G_0$ and $Q$.

\begin{theorem}\label{t:Q-structure}
Let $(X,\omega)$ be a symplectic rational $G$-surface equipped with a minimal symplectic 
$G$-conic bundle structure with at least three singular fibers (i.e, $N\geq 4$). Let $G_0,Q$ 
be given as in Definition \ref{d:QandG0}.
Then one of the following is true.
\begin{itemize}
\item [{1.}] $G_0=\Z_m$, $m>1$, and $Q$ is either the dihedral group $D_{2m}$ 
containing $G_0$ as an index $2$ subgroup,
or $Q=G_0$ and $m$ is even. Moreover, $N$ must be odd, and any element 
$\tau\in Q\setminus G_0$ switches the two $(-1)$-spheres in each singular fiber.
\item [{2.}] $G_0$ is trivial and $Q=\Z_2$ or $(\Z_2)^2$. In the latter case, let 
$\tau_1,\tau_2,\tau_3$ be the
distinct involutions in $Q$. Then $\Sigma$ is partitioned into subsets $\Sigma_1$, $\Sigma_2$,
$\Sigma_3$, where $\Sigma_i$ parametrizes those singular fibers of which $\tau_i$
leaves each $(-1)$-sphere invariant, and $\# \Sigma_i\equiv N-1 \pmod{2}$,  for $i=1,2,3$.
\end{itemize}
\end{theorem}

\begin{remark}\label{rem:compareQ}
\begin{itemize}
\item [{(a)}]
Note that when $G_0$ is trivial and $Q=(\Z_2)^2$, each element of $P$ acts on $Q$ as an 
automorphism, permuting the three involutions $\tau_1,\tau_2,\tau_3$. Consequently, 
the action of $P$ on the base $\s^2$ preserves the partition $\Sigma=\Sigma_1\sqcup\Sigma_2\sqcup\Sigma_3$. 
For the corresponding result in algebraic geometry, see \cite{DI},
Theorem 5.7.

\item [{(b)}] When the fiber class of the symplectic $G$-conic bundle is unique, $Q$ is uniquely
determined as a subgroup of $G$, see Proposition 4.5 for more details. 
\end{itemize}
\end{remark}

\subsection{Minimality and equivariant symplectic cones}

In this subsection, we are concerned with the underlying smooth action of a minimal symplectic rational $G$-surface. In particular, our consideration here offers an interesting symplectic geometry
perspective to the study of rational $G$-surfaces in algebraic geometry.

We begin with the setup of our study here. Let $X=\C\P^2\# N\overline{\C\P^2}$, $N\geq 2$,
which is equipped with a smooth action of a finite group $G$. Suppose there is a $G$-invariant
symplectic form $\omega_0$ on $X$ such that the corresponding symplectic rational 
$G$-surface $(X,\omega_0)$ is minimal. With this understood, we denote
\begin{equation}\label{e:omegaXG}
      \Omega(X,G):=\{\w:\w\text{ is symplectic on }X, g^*\w=\w,\text{ for any }g\in G\}.
 \end{equation}

Note that $\Omega(X,G)$
is non-empty as $\omega_0\in\Omega(X,G)$. The part (2) of the following theorem shows 
that the underlying smooth action of a minimal (complex) rational $G$-surface satisfies the above 
assumption, where we can take $\omega_0$ to be any $G$-invariant K\"{a}hler form.

\begin{theorem}\label{t:minimality}
(1) Let $X=\C\P^2\# N\overline{\C\P^2}$, $N\geq 2$, which is equipped with a smooth action 
of a finite group $G$. Suppose there is a $G$-invariant symplectic form $\omega_0$ on $X$ 
such that $((X,\omega_0),G)$ is minimal. Then for any $\omega\in \Omega(X,G)$, 
the canonical class $K_\omega=K_{\omega_0}$ or $-K_{\omega_0}$, and the symplectic rational $G$-surface 
$(X,\omega)$ is minimal. 

(2) Let $X$ be any minimal (complex) rational $G$-surface which is $\C\P^2$ blown up at $2$ or
more points. Then for any symplectic form $\omega$ which is invariant under the underlying
smooth action of $G$ (e.g., any $G$-invariant K\"{a}hler form), the corresponding symplectic 
rational $G$-surface $(X,\omega)$ is minimal.
\end{theorem}

At the time of writing, it is not known whether
the symplectic minimality would imply the smooth minimality of the underlying group action.  The symplectic minimality in Theorem \ref{t:minimality} is a weaker statement that symplectic minimality 
is determined by the underlying smooth action, but the proof is still quite non-trivial.
For related (stronger) results in the case of $G$-Hirzebruch surfaces, see \cite{C4}.

\vskip 5mm
With the minimality in Theorem \ref{t:minimality}(1) in place, we now turn our attention to the equivariant 
symplectic cone of the $G$-manifold $(X,G)$.

\begin{defn}\label{d:omegaXG} The equivariant symplectic cone of $(X,G)$ is defined as
    $$\tilde{C}(X,G)=\{\Omega:\Omega=[\w],\textit{ }\w\textit{ is a $G$-invariant symplectic form}\}
    \subset H^2(X;\R)^G.$$
\end{defn}

Note that $K_{-\w}=-K_\w$. Therefore it suffices to
consider the subset 
$$
C(X,G):=\{[\omega]| \omega\in \Omega(X,G), K_\omega=K_{\omega_0}\}\subset H^2(X;\R)^G.
$$
Furthermore, we observe that if $H^2(X)^G$ has rank $1$, 
$C(X,G)=\{\lambda K_{\omega_0}|\lambda\in \R, \lambda<0\}$.
In what follows, we shall assume that $H^2(X)^G$ has rank $2$. Note that under this assumption, 
$N\neq 6$ by part (2) of Theorem \ref{t:X-structure}.

In order to describe $C(X,G)$, it is helpful to introduce the following terminology. A class
$F\in H^2(X)^G$ is called a {\bf fiber class} if there exists an $\omega\in  \Omega(X,G)$
such that $F$ is the class of the regular fibers of a symplectic $G$-conic bundle on
$((X,\omega),G)$. Since we will focus on the set $C(X,G)$, we shall assume further that
$[\omega]\in C(X,G)$, i.e., $K_\omega=K_{\omega_0}$. 

We observe that since $\text{rank }H^2(X)^G=2$, the class of such an $\omega$ can be written
as 
$$[\omega]=-a K_{\omega_0}+bF, a>0.$$ 

With this understood, we consider the following subset of $C(X,G)$ and its projective classes
\begin{equation}\label{e:}
   \begin{aligned}
           &C(X,G,F)=\{[\w]\in C(X,G): \w=-aK_{\w_0}+bF, a>0,b\ge0\},\\
      &\hat C(X,G,F)=\{[\w]\in C(X,G,F): \w(F)=2\} \text{ (equivalently, $a=1$)}.   
   \end{aligned}
\end{equation}



Now $[\w]\in \hat C(X,G,F)$ can be written as $[\omega]=-K_{\omega_0}+
\delta_{\omega,F} F$. 
Then the one to one correspondence 
$[\omega]\mapsto \delta_{\omega,F}$ identifies $\hat{C}(X,G,F)$ with a subset of $\R$. With this
understood, we introduce 
$$
\delta_{X,G,F}:=\inf_{[\omega]\in \hat{C}(X,G,F)}\delta_{\omega,F}\in [0,\infty).
$$
Note that $(X,\omega)$ is monotone if and only if $\delta_{\omega,F}=0$, so $\delta_{\omega,F}$
may be thought of as a sort of gap function which measures how far away $(X,\omega)$ is from
being monotone. 

\begin{theorem}\label{t:sympCone}  
Let $X=\C\P^2\# N\overline{\C\P^2}$, $N\geq 2$, be equipped with a smooth finite $G$-action
which is symplectic and minimal with respect to some symplectic form $\omega_0$. Furthermore,
assume $\text{rank }H^2(X)^G=2$.

\begin{enumerate}[(1)]
  \item If $N\geq 9$ or $G_0$ is nontrivial, there is a unique fiber class $F$, and $C(X,G)=C(X,G,F)$. 
  \item For $N=5,7,8$, either there is a unique fiber class $F$, or there are two distinct fiber 
classes $F,F^\prime$. In the former case, $C(X,G)=C(X,G,F)$, and in the latter case,
$C(X,G)=C(X,G,F)\cup C(X,G,F^\prime)$, with $C(X,G,F)\cap C(X,G,F^\prime)$ being either 
empty or consisting of $[\omega]$ such that $(X,\omega)$ is monotone.

\item For any fiber class $F$, $\hat{C}(X,G,F)$ is identified with either
$[0,\infty)$ or $(\delta_{X,G,F},\infty)$ under $[\omega]\mapsto \delta_{\omega,F}$. 
(In particular, $\delta_{X,G,F}$ can not be attained unless it equals $0$.)
\end{enumerate}

\end{theorem}

We conjecture that when there are two distinct fiber classes, the equivariant symplectic cone 
must contain the class of a monotone form. Furthermore, it is an interesting problem to determine
the gap functions $\delta_{X,G,F}$ for a given minimal rational $G$-surface $X$ with 
$\text{Pic}(X)^G=\Z^2$. We shall leave these studies for a future occasion. 

\vspace{3mm}

The organization of this paper is as follows. Section 2 is concerned with the proof of the structural
theorem, Theorem 1.3. In section 2.1 we collect some preliminary lemmas on minimal symplectic 
$G$-conic bundles. In section 2.2 we review a reduction process for the exceptional classes 
of a rational surface which plays an essential role in the proof of Theorem 1.3. Sections 2.3
and 2.4 are devoted to the proof of Theorem 1.3. Section 3 is concerned with the analysis of
the structure of group $G$. Proofs of Theorems 1.5 and 1.8 are presented here. Finally,
section 4 is devoted to the discussion on equivariant symplectic minimality and equivariant symplectic cones. In particular, we prove Theorems 1.10 and 1.12. At the end of section 4,
we also include a uniqueness result on the subgroup $Q$ in Definition 1.7. 

\vspace{3mm}

\noindent{\bf Acknowledgments:}

This long overdue project started from the \textit{Workshop and Conference on Holomorphic Curves and Low Dimensional Topology} in Stanford, 2012.
We are grateful to Igor Dolgachev for inspiring conversations during the \textit{FRG conference on symplectic birational geometry 2014} in University of Michigan.  This work was partially supported by NSF Focused Research Grants DMS-0244663 and DMS-1065784, under \textit{FRG: Collaborative Research: Topology and Invariants of Smooth $4$-Manifolds}.


\section{The structure of $(X,\omega)$}

\subsection{Preliminary lemmas on symplectic $G$-conic bundles}

Let $(X,\omega)$ be a symplectic rational $G$-surface, where 
$X=\C\P^2\# N\overline{\C\P^2}$, and let $\pi: X\rightarrow \s^2$ be
a symplectic $G$-conic bundle on $(X,\omega)$. We first observe that 
each singular fiber of $\pi$ consists of a pair of $(-1)$-spheres. To see this,
let $C_1,C_2$ be the components of a singular fiber, which are embedded 
$J$-holomorphic spheres. Since $C_1,C_2$ are isolated $J$-curves, their self-intersection
must be negative. On the other hand, $C_1\cdot C_2=1$, so it follows easily from
$(C_1+C_2)^2=0$ that both of $C_1,C_2$ are $(-1)$-spheres.

There are $N-1$ singular fibers. We shall pick a $(-1)$-sphere from each singular fiber
and name the homology classes by $E_2,\cdots, E_N$. Then there is a unique pair of line class and exceptional class $H$ and  $E_1$
such that 

\begin{enumerate}[(i)]
  \item $H-E_1$ is the fiber class of $\pi$,
  \item $H,E_1, E_2,\cdots, E_N$ form a basis
of $H_2(X)$ with standard intersection matrix.
\end{enumerate}

 With this understood, observe that

 \begin{enumerate}[(1)]
   \item for each $(-1)$-sphere $E_j$, $j=2,\cdots, N$, the other $(-1)$-sphere lying in the same
singular fiber has homology class $H-E_1-E_j$,
   \item the canonical class $K_\omega=-3H+
E_1+\cdots +E_N$.
 \end{enumerate}

There are further consequences if the symplectic $G$-conic bundle is minimal. In this case, 
for each $(-1)$-sphere $E_j$, $j=2,\cdots,N$, there exists a $g\in G$ such that 
$g\cdot E_j=H-E_1-E_j$. This implies
$$
\omega(E_j)=\frac{1}{2}\omega(H-E_1), \;\; j=2,\cdots, N.
$$

Thus a minimal symplectic $G$-conic bundle falls into three cases:
$$
(i) \; \omega(E_1)=\omega(E_j), \; (ii) \; \omega(E_1)>\omega(E_j), \; (iii)\;  \omega(E_1)<\omega(E_j).
$$
Case (i) occurs iff $(X,\omega)$ is monotone. Since $E_1$ is a section class, we shall call
case (ii) (resp. case (iii)) a symplectic $G$-conic bundle with \textit{small fiber area} (resp. \textit{large
fiber area}).

The following lemma is the $J$-holomorphic analog of Lemma 5.1 in \cite{DI}. 

\begin{lemma}\label{l:selfint}
Let $\pi:X\rightarrow \s^2$ be a symplectic $G$-conic bundle on $(X,\omega)$,
which comes with a $G$-invariant, $\omega$-compatible almost complex structure $J$.
Suppose $E, E^\prime$ are two distinct $J$-holomorphic sections of self-intersection 
$-m,-m^\prime$, and let $r$ be the number of singular fibers where $E,E^\prime$ intersect 
the same component. Then
$$
N-1=r+m+m^\prime+2E\cdot E^\prime.
$$
Moreover, if the symplectic $G$-manifold $(X,\omega)$ is minimal, then $N\geq 5$ must be true.
\end{lemma}

\begin{proof}

Since $E_1,F=H-E_1,E_s,s>1$ generate $H^2(X)$, both $E$ and $E'$ have the form 
\begin{equation}\label{e:decomp}
     E_1+cF+\sum_{t>1}c_tE_t,
\end{equation}
where $c\in\Z$, and $c_t=0$ or $1$ depending on which component they intersect on each singular fiber.

Since $E,E^\prime$ are sections, we have
 $$
  E^\prime-E=bF +\sum_s b_s E_s,
 $$
 where $b\in\Z$ and $b_s=\pm 1$ with $s$ running over the set of singular fibers at where
 $E^\prime, E$ intersect \textit{different} components. Note that the number of $s$ is exactly $N-1-r$.
 Now 
 $$(E^\prime-E)^2= \sum_s b_s^2 E_s^2=-(N-1-r),$$
 which gives
 $$
 N-1=r+m+m^\prime+2E\cdot E^\prime.
 $$

Now suppose $(X,\omega)$ is minimal. We will show that $N\geq 5$ in this case.
First, we claim that the sections $E,E^\prime$ in the lemma do exist. This is because 
the class $E_1$ can be represented by a $J$-holomorphic stable curve.  By checking the intersection with $H-E_1$, the $E_1$ stable representative contains a unique  $J$-holomorphic section $E$.  Note that $E,F=H-E_1,E_s,s>1$ also form a basis of $H_2(X)$, and 
$E_1$ can be written as
 $$
 E_1=a(H-E_1)+\sum_{s>1} a_s E_s+E,
 $$ 
 where $a\geq 0$, and $a_s=0$ or $1$ depending on which component that $E$ intersects at each singular fiber. 
 Note that 
 $$
 E^2=-2a-1-\sum_{s>1} a_s.
 $$
 We take $E^\prime=g\cdot E$ for some $g\in G$ such that $E^\prime\neq E$. Such a $g\in G$ exists 
 because $E$ must intersect a singular fiber and there is a $g\in G$ which switches the two 
 $(-1)$-spheres in that singular fiber. Note that $E^\prime$ is a section, as the fiber class is $G$-invariant so that $E^\prime \cdot (H-E_1)=E\cdot(H-E_1)=1$. Furthermore, note that $(E^\prime)^2=E^2$. Consequently, if $E^2\leq -2$, we must have
 $$
 N=1+r+m+m^\prime+2 E\cdot E^\prime\geq 1+0+2+2+0=5.
 $$
 If $E^2>-2$, then $a=a_s=0$ and $E=E_1$ must be a $(-1)$-sphere. The minimality
 assumption then implies that $E^\prime=g\cdot E$ for some $g\in G$ can be chosen such that 
 $E^\prime$ intersects $E$. In this case,  we have 
 $$
 N= 1+r+m+m^\prime+2 E\cdot E^\prime\geq 1+0+1+1+2=5.
 $$
 This finishes off the proof of Lemma \ref{l:selfint}.
 
 \end{proof}

 \begin{lemma}\label{l:ConicRank}
 Let $(X,\omega)$ be a symplectic rational $G$-surface which admits a minimal
 symplectic $G$-conic bundle structure. Then the invariant lattice $H^2(X)^G$ has
 rank $2$ which is spanned by $K_\omega$ and the fiber class of the  
 symplectic $G$-conic bundle.
  \end{lemma}
 
 \begin{proof}
 
 First, we show that $H^2(X;\R)^G$ is $2$-dimensional.
To see this, we first note that  $K_\omega, H-E_1, E_2,\cdots,E_N$ form a basis of 
$H^2(X,\R)$.  We set 
 $$
 V=H^2(X;\R)/\text{Span }_\R(K_\omega,H-E_1).
 $$
 Then it suffices to show that $V^G=\{0\}$ because $K_\omega,H-E_1\in H^2(X)^G$.
 
 We let $e_2,\cdots,e_N$ be the image of $E_2,\cdots,E_N$ under the quotient map, which
 form a basis of $V$. Suppose to the contrary, there is a $x\neq 0$ in $V^G$. We write
 $x=\sum_{k=2}^N a_k e_k$. Then there exists a $k_0$ such that $a_{k_0}\neq 0$.
 With this understood, we note that by the minimality assumption,
 there is a $g\in G$ such that $g\cdot E_{k_0}=H-E_1-E_{k_0}$,
 which means that $g\cdot e_{k_0}=-e_{k_0}$. Now we set $I=\{k| g\cdot e_k=-e_k\}$. Then
 $k_0\in I$; in particular, $I\neq \emptyset$. We let $J$ be the complement of $I$ in the set
 $\{2,\cdots,N\}$. Then it follows easily that if $k\in J$, then $g\cdot e_k=\pm e_l$ for some 
 $l\in J$ ($g\cdot E_k=E_l$ or $H-E_1-E_l$ for some $l$).
 
 With this understood, we write $x=y+z$, where $y=\sum_{k\in I} a_k e_k$ and 
 $z=\sum_{k\in J} a_ke_k$. Then $g\cdot x=-y+z^\prime$ for some $z^\prime\in 
 \text{Span }_\R(e_k|k\in J)$. Since $g\cdot x=x$, we have 
 $2y=z^\prime-z\in \text{Span }_\R(e_k|k\in J)$. Since $e_2,\cdots,e_N$ form a basis of $V$,
 this clearly contradicts the fact that $y\in \text{Span }_\R(e_k|k\in I)$ and $y\neq 0$. Hence the 
 claim that $H^2(X;\R)^G$ is $2$-dimensional.

 Now for any $\alpha\in H^2(X)^G$, we write $\alpha=aK_\w+b (H-E_1)$ for some $a,b\in\Q$.
 Then $\alpha\cdot E_N=-a$, implying $a\in \Z$. On the other hand, $\alpha\cdot E_1=-a+b$,
 we have $b\in \Z$ also. Hence $H^2(X)^G$ is spanned by $K_\w$ and $H-E_1$.

\end{proof}

\begin{lemma}\label{4-3}
Let $\pi: X\rightarrow\s^2$ be a minimal symplectic $G$-conic bundle where $X=\C\P^2
\# N\overline{\C\P^2}$ with $N\geq 6$ and even. Let $J$ be any $G$-invariant, compatible
almost complex structure, and let 
$$m_J=max\{m\in \Z: \textit{there is a $J$-holomorphic section of $\pi$ of self-intersection }-m\}.$$

 Then $m_J\leq (N-4)/2$.  
In particular, when $N=6$, $m_J\leq 1$.
 \end{lemma}
 
 \begin{proof}
 First, note that if $E$ is a $J$-holomorphic section of self-intersection $-m$, then there is
 a $g\in G$ such that $g\cdot E\neq E$, where $g\cdot E$ is also of self-intersection $-m$.
 This is because $E$ must intersect a singular fiber, and by the minimality assumption there
 is a $g\in G$ which switches the two components of that singular fiber. Clearly for this $g$,
 $g\cdot E\neq E$. Now
 apply Lemma  \ref{l:selfint} to $E$ and $E^\prime=g\cdot E$, we see that $m\leq (N-2)/2$
 as $N$ is even. Hence we reduced the lemma to showing $m\neq \frac{N-2}{2}$.
 
 Suppose to the contrary that $m=(N-2)/2$, and let $E,E^\prime$ be a pair of sections 
 whose self-intersection equals $-m$. Then by Lemma \ref{l:selfint}, we see that $E,E^\prime$ are disjoint and $r=1$. Let $F$ be the singular fiber where $E,E^\prime$ intersect the
 same $(-1)$-sphere. Then again by the minimality assumption there is a $h\in G$ which 
 switches the two $(-1)$-spheres in $F$. It follows easily that $h\cdot E$, $E$,
 and $E^\prime$ are distinct. Let $r,r^\prime$ be the number of singular fibers where 
 $h\cdot E$, $E$ and $h\cdot E$, $E^\prime$ intersect the same $(-1)$-sphere, respectively. 
 Then it follows easily that $r+r^\prime=N-2$. On the other hand, Lemma \ref{l:selfint} implies that $r=r^\prime=1$,
 contradicting the fact that $N\geq 6$. Hence the lemma.
 
 \end{proof}

\subsection{Reduction of exceptional classes} 
\label{sub:reduction_of_exceptional_classes}


Momentarily we let $\omega$ be any symplectic structure on $X=\C\P^2\# N\overline{\C\P^2}$, 
and denote 
 $$\E_\omega=\{e\in H^2(X)| e^2=-1, K_\omega\cdot e= -1, \omega(e)>0\}$$
 which may depend on $\omega$. The following fact is crucial in our considerations.

\begin{lemma}\label{l:Pinsonnault}
(cf. \cite{KK})
Assume $N\geq 2$. Then for any $\omega$-compatible almost complex structure $J$ on $X$, each class $E\in\E_\omega$ 
with minimal
area, i.e., $\omega(E)=\min_{e\in\E_\omega} \omega(e)$, is represented by an {\it embedded}
$J$-holomorphic sphere.
\end{lemma}

The key ingredient in the proof of Theorem \ref{t:X-structure} is a reduction procedure which involves a
certain type of standard basis of $H^2(X)$, called a {\it reduced basis}. We shall begin by
a brief digression and refer the reader to \cite[Proposition 4.14]{LW12} or independently \cite{BP} for more details.


Recall that a {\it reduced basis} is a basis $H, E_1,\cdots, E_N$ of $H^2(X)$ with standard
intersection matrix, where $E_i\in\E_\omega$, such that $\omega(E_N)=\min_{e\in\E_\omega}\omega(e)$, and for any $i<N$, $E_i$ satisfies the following inductive condition: let 
$\E_i=\{e\in\E_\omega| e\cdot E_j=0 \;\; \forall i<j\}$, then $\omega(E_i)=\min_{e\in\E_i}\omega(e)$. Furthermore, the canonical class $K_\omega=-3H+E_1\cdots+E_N$. Reduced basis
always exists. 

If $N=2$, then $\E_\omega=\{E_1,E_2, H-E_1-E_2\}$. For $N\geq 3$, a reduction procedure can be
introduced, which requires the following discussions.

\vspace{2mm}

1. Introduce $H_{ijk}=H-E_i-E_j-E_k$ for $i<j<k$ and $H_{ij}=H-E_i-E_j$ for $i<j$. Then
$H_{ij}\in\E_j$, which implies that $\omega(H_{ijk})\geq 0$. 

\vspace{2mm}

2. For any $E\in \E_\omega$, write $E=aH-\sum_s b_s E_s$. Then
\begin{itemize}
\item $a\geq 0$;
\item if $a>0$, then $b_s\geq 0$ for all $s$;
\item if $a=0$, then $E=E_l$ for some $l$;
\item assume $a>0$, and let $b_i,b_j,b_k$ be the largest three coefficients (here we use the assumption $N\geq 3$), 
then $b_i\le a<b_i+b_j+b_k$, which is equivalent to $E\cdot H_{ijk}<0$ and $E\cdot(H-E_i)\ge0$.
\end{itemize}

\vspace{2mm}

3. The classes $H_{ijk}$ are represented by embedded $(-2)$-spheres, hence for each $H_{ijk}$
there is a diffeomorphism of $X$ inducing an automorphism $R(H_{ijk})$ on $H^2(X)$:
$$
R(H_{ijk})\alpha =\alpha + (\alpha\cdot H_{ijk})\cdot H_{ijk}, \;\; \forall \alpha\in H^2(X).
$$
Moreover, each $R(H_{ijk})$ has the following properties: (1) $R(H_{ijk})K_\omega=K_\omega$,
(2) $R(H_{ijk})E\in \E_\omega,\;\; \forall E\in\E_\omega$. 

\vspace{2mm}

With the preceding understood, consider any $E\in\E_\omega$, where $E=aH-\sum_s b_s E_s$ 
with $a>0$. Let $b_i,b_j,b_k$ be the largest three coefficients for some $i<j<k$. Then
it follows easily from the last bullet of item 2 that 
\begin{itemize}
\item $R(H_{ijk})E=a^\prime H-\sum_s b_s^\prime E_s$ for some $a^\prime<a$, and 
\item $\omega (R(H_{ijk})E)\leq \omega(E)$ with $``="$ iff $\omega(H_{ijk})=0$.
\end{itemize}
Set $E^\prime:= R(H_{ijk})E$. We say that {\it $E$ is reduced to $E^\prime$ by $H_{ijk}$}. 

The operations $R(H_{ijk})$ has the following properties of our interests:

\begin{enumerate}
   \item (finite termination) By \cite[Proposition 1.2.12]{MS09}, one may find a sequence of finitely many $H_{ijk}$ for any $E\in \E_\omega$, such that after performing $R(H_{ijk})$, $E$ is reduced to $E_l$ for some $l$.
   \item (monotonicity) The symplectic area is \textbf{monotonically decreasing} during the above reduction procedure.
 \end{enumerate}

Therefore, after the reduction procedure, $\omega(E)\geq\omega (E_l)$, with $``="$ iff 
$\omega(H_{ijk})=0$ for all the
$H_{ijk}$'s involved. In particular, when $E$ has the minimal area in $\E_\omega$, i.e., 
$\omega(E)=\omega(E_N)$, then we have

\begin{equation}\label{e:Eequal}
     \omega(E_l)=\omega(E_{l+1})=\cdots =\omega(E_N),
\mbox{ and }\omega(H_{ijk})=0 \mbox{ for all the } H_{ijk}'s \mbox{ involved.}
\end{equation}

We first rule out the case of $N=2$.

\begin{lemma}\label{l:at least 3}
There are no minimal symplectic rational $G-$surface with $N=2$. 
\end{lemma}

\begin{proof}
Suppose $((X, \omega), G)$ is a minimal symplectic rational $G-$surface with  $N=2$. 
Let $\{H,E_1,E_2\}$ be a reduced basis. Then
$\E_\omega=\{E_1,E_2,H-E_1-E_2\}$. 

Fix a $G$-invariant $J$, and let $C$ be the $J$-holomorphic $(-1)$-sphere representing $E_2$
(cf. Lemma \ref{l:Pinsonnault}). We set $\Lambda:=\cup_{g\in G} g\cdot C$. Then $\Lambda$ is a union of
finitely many $J$-holomorphic $(-1)$-spheres, containing at least two distinct $(-1)$-spheres 
intersecting each other because of the minimality assumption. Since 
$\E_\w=\{E_1,E_2,H-E_1-E_2\}$, 
there are only two possibilities: (1) $\Lambda$ is a union of three $(-1)$-spheres, 
representing the classes
$E_1,E_2$ and $H-E_1-E_2$, and (2) $\Lambda$ is a union of two $(-1)$-spheres representing
the classes $E_2$ and $H-E_1-E_2$. In either case, $H-E_1-E_2$ is represented by
a $(-1)$-sphere. Since $H-E_1-E_2$ is the only characteristic element in $\E_\w$, it
must be fixed by the $G$-action, which contradicts the minimality of the symplectic 
$G$-manifold $(X,\omega)$. Hence   there are no minimal symplectic rational $G-$surface with $N=2$. 

\end{proof}

\subsection{Reduced basis for symplectic rational $G$-surfaces} 
\label{sub:reduced_basis_for_symplectic_}


\vspace{3mm}

Due to Lemma \ref{l:at least 3}, in what follows we assume that $(X,\omega)$ is a minimal symplectic rational $G$-surface, 
where $X=\C\P^2\# N\overline{\C\P^2}$ with $N\geq 3$. 

\begin{lemma}\label{l:structureLemma}
Suppose $N\geq 3$. Then one of the following must be true.
\begin{itemize}
\item [{(i)}] $(X,\omega)$ is monotone. 
\item [{(ii)}] The reduced basis of $X$ satisfies $\omega(E_1)>\omega(E_2)=\cdots =\omega(E_N)$, $\w(H-E_1)=2\w(E_j)$. 
Moreover, if $E\in \E_\omega$ has
minimal area, then either $E=E_j$ or $E=H_{1j}=H-E_1-E_j$  for some $j>1$.
\end{itemize}
\end{lemma}

\begin{proof}
Assume $(X,\omega)$ is not monotone. We shall first prove a slightly weaker statement that is independent of the $G$-action.
\vskip 5mm

\noindent{\bf Claim:} If $E\in\E_\omega$ has 
minimal area
in $\E_\omega$ and $E\neq E_s$ for any $s$, then $E=H-E_1-E_j$ for some $j>1$, and furthermore,
if such an $E$ exists and $E=H-E_1-E_j$ for some $j>2$, then we must have
$$
\omega(E_2)=\cdots =\omega(E_N).
$$

\begin{proof}[Proof of Claim] 
Let $E\in\E_\omega$ be such a class, i.e., $E$ has minimal area in $\E_\omega$ and 
$E\neq E_s$ for any $s$,
We reduce $E$ to $E_l$ for some $l$ by a sequence of $H_{ijk}$'s.
Since $E$ has minimal area in $\E_\omega$, it follows that $\omega(E_l)=\cdots=\omega(E_N)$.
Note that $l>1$.  Otherwise, $\w(E_i)=\w(E_j)$ and $\w(H_{ijk})=0$ for some $i,j,k$ from property (3) in the reduction process in Section \ref{sub:reduction_of_exceptional_classes}.  This violates that $(X,\omega)$ is not monotone. 

Suppose the reduction from $E$ to $E_l$ takes $n$ steps, let $E^\prime\in\E$ be the class 
obtained at the $(n-1)^{\text{th}}$ step.
Then $E^\prime=aH -\sum_s b_s E_s$, where $a>0$ and $b_s\geq 0$. 
The equation $E_l=E^\prime+ (E^\prime \cdot H_{ijk})\cdot H_{ijk}$ reads 
$$
E_l=(2a-b_i-b_j-b_k)H -\sum_{s=i,j,k}(b_s+a-b_i-b_j-b_k)E_s-\sum_{s\neq i,j,k} b_sE_s,
$$
which implies that $l$ must be one of $i,j$ or $k$, and without loss
of generality assuming that $l=k$, then 
$$
2a=b_i+b_j+b_k,\; a=b_j+b_k=b_i+b_k, \;a-b_i-b_j=-1.
$$
It follows easily that $E^\prime=H_{ij}=H-E_i-E_j$. Here we assume  $i<j$, but do not require 
any condition on $k=l>1$. Note that $E^\prime$ has minimal area in $\E_\omega$, which implies that
$\omega(H)=\omega(E_i)+\omega(E_j)+\omega(E_l)$ from \eqref{e:Eequal}. 

Next we prove that $i=1$. Suppose to the contrary that $i>1$. Then
$\omega(H_{1ij})=\omega(H)-\omega(E_1)-\omega(E_i)-\omega(E_j)\geq 0$. On the 
other hand, $\omega(H)=\omega(E_i)+\omega(E_j)+\omega(E_l)$ and $\omega(E_1)\geq
\omega(E_l)$, which implies that $\omega(E_1)=\cdots =\omega(E_N)=\frac{1}{3}\w(H)$. This is a contradiction
because we assume $(X,\omega)$ is not monotone. Hence $i=1$ and $E^\prime=H-E_1-E_j$.

We claim that $E=E^\prime=H-E_1-E_j$. Suppose this is not true. Then there must be a class 
$\tilde{E}$ which is reduced to $E^\prime$ by some $H_{vrt}$ for $v<r<t$. We assert $v>1$. To see this,
write $\tilde{E}=aH -\sum_s b_s E_s$, where $a>0$ and $b_s\geq 0$. Then similarly we have
$$
H-E_1-E_j=(2a-b_v-b_r-b_t)H -\sum_{s=v,r,t}(b_s+a-b_v-b_r-b_t)E_s-\sum_{s\neq v,r,t} b_sE_s.
$$
If $v=1$, then $2a-b_v-b_r-b_t=1$ and $a-b_r-b_t=1$, implying $a=b_v$. Now for $s=r,t$,
the coefficient of $E_s$ on the right hand side is $b_t$, $b_r$ respectively, which are 
non-negative. It follows that $b_r=b_t=0$ from the property 2 in Section \ref{sub:reduction_of_exceptional_classes}, which contradicts the fact that $a<b_v+b_r+b_t$.
Hence $v>1$. 

We now get a contradiction as follows. Note that $\omega(H_{vrt})=0$ and
$\omega(H_{1vr})\geq 0$ implies that 
$$
\omega(E_1)=\cdots=\omega(E_t).
$$
Then with $\omega(H_{vrt})=0$ again, we have $\omega(H)=3\omega(E_1)$. On the other
hand, $\omega(E_l)=\omega(E^\prime)=\omega(H)-\omega(E_1)-\omega(E_j)$, from which it follows that
$$
\omega(E_1)=\cdots=\omega(E_l)=\cdots=\omega(E_N).
$$
This contradicts the assumption that $(X,\omega)$ is not monotone. Hence $E=H-E_1-E_j$
is proved. Finally, if $j>2$, then $\omega(H_{12j})\geq 0$ and $\omega(H)-\omega(E_1)-\omega(E_j)=\omega(E_l)$ implies that $\omega(E_2)=\omega(E_l)$, hence 
$\omega(E_2)=\cdots=\omega(E_N)$. This concludes the claim.

\end{proof}

To obtain Lemma \ref{l:structureLemma}, it remains to show that if $(X,\omega)$ is not monotone, then there exists some $j>2$, such that $E=H-E_1-E_j$
attains the minimal area within $\mathcal{E}_\w$. To this end, we fix a $G$-invariant $\omega$-compatible $J$. 
Then by Lemma \ref{l:Pinsonnault}, there exists a $J$-holomorphic $(-1)$-sphere $C$ representing the class 
$E_N$. Note that for any $g\in G$, $g\cdot C\in\E_\omega$. Now by the assumption that 
$(X,\omega)$ is minimal, there must be a $g\in G$ such that $g\cdot C\neq C$ and $g\cdot C$ intersects with $C$. 
The class $g\cdot C$ has minimal area in $\E_\omega$ and $g\cdot C\neq E_s$ for any $s$. 
Let $g\cdot C=H-E_1-E_j$. Then $j=N\geq 3$ must be true because $g\cdot C$ intersects 
with $C$. This finishes the proof of the lemma.

\end{proof}

\subsection{Proof of Theorem \ref{t:X-structure}} 
\label{sub:proof_of_theorem_1_1}


\vspace{3mm}

In what follows, we still assume that $N\geq 3$. We will discuss according to the following three possibilities:

\vspace{2mm}

i) {\it $(X,\omega)$ is not monotone}: By Lemma \ref{l:structureLemma}, there is a reduced basis 
$\{H, E_1,\cdots,E_N\}$ such that 
$$
\omega(E_1)>\omega(E_2)=\cdots=\omega(E_N),
$$
and moreover, if $E\in \E_\omega$ has
minimal area, then either $E=E_j$ or $E=H_{1j}=H-E_1-E_j$  for some $j>1$.

Let $J$ be any $G$-invariant $\omega$-compatible almost complex structure. For each fixed $j>1$, $E_j$ has minimal area so
by Lemma \ref{l:Pinsonnault}, so there is an embedded $J$-holomorphic sphere $C$ representing $E_j$. Since 
$(X,\omega)$ is minimal,
there must be a $g\in G$ such that $g\cdot C\neq C$ and $g\cdot C\cap C\neq \emptyset$. It follows easily that $g\cdot C$ must
be the $J$-holomorphic $(-1)$-sphere representing the class $H-E_1-E_j$, since it also has the minimal area. Furthermore,
$g\cdot C$ and $C$ intersect transversely and positively at a single point. Standard gluing construction in $J$-holomorphic curve
theory yields a $J$-holomorphic sphere $\hat{C}$, carrying the class $C+g\cdot C=H-E_1$. 
It follows that $\hat{C}$ has self-intersection
$0$, and by the adjunction formula it must be embedded. By a standard Gromov-Witten index computation, the moduli space of $J$-spheres in class $[\hat C]$ gives
rise to a fibration (which contains singular fibers) structure on $X$, where each fiber is homologous to $\hat{C}$. Since $X$ is a rational surface,
the base of the fibration must be $\s^2$. We denote the fibration by $\pi_j: X\rightarrow \s^2$. 

Next we show that $\pi_j$ is $G$-invariant. To see this, note that for any $h\in G$, $h\cdot (C\cup g\cdot C)$ must be a pair of
$J$-holomorphic $(-1)$-spheres representing $E_k$ and $H-E_1-E_k$ for some $k$ 
(see Lemma \ref{l:structureLemma}). It follows that the class
of $\hat{C}$, which is $H-E_1$, is invariant under the $G$-action. This implies that the fibration 
$\pi_j$ is $G$-invariant. 

Finally, we note that the fibration $\pi_j$ is independent of $j$, because the fiber class, which is 
$H-E_1$, is independent of $j$.
We will denote the fibration by $\pi: X\rightarrow \s^2$. 
Note that the same argument shows also that $\pi$ contains at least $N-1$ singular fibers, 
consisting of a pair of $(-1)$-spheres whose classes are $E_j, H-E_1-E_j$ for $j=2,\cdots, N$. 
There are no other singular fibers by an Euler number count.  Note that Lemma \ref{l:ConicRank} asserts $\text{rank }H^2(X)^G=2$ in this case.

\vspace{2mm}

ii) {\it $(X,\omega)$ is monotone and $H^2(X)^G$ has rank $1$}: In this case, note that 
$K_\omega\in H^2(X)^G$ and is a primitive class, hence $H^2(X)^G$ is spanned by 
$K_\omega$. The constraint $N\leq 8$ is an easy consequence of $(X,\omega)$ being monotone.  

\vspace{2mm}

iii) {\it $(X,\omega)$ is monotone and $\text{rank } H^2(X)^G>1$}: Note that the de Rham class 
$[\omega]=\lambda K_\omega\in H^2(X;\R)^G$ for some $\lambda\in\R$. 
Since $\text{rank } H^2(X)^G>1$, we may pick a $G$-invariant closed $2$-form $\eta$ such that $[\eta]$ lies in a different direction
in $H^2(X;\R)^G$. Let $\omega^\prime:=\omega+\epsilon \eta$ for some very small $\epsilon$. Then $\omega^\prime$ is a 
$G$-invariant symplectic structure such that $(X,\omega^\prime)$ is not monotone. \\

\noindent\textbf{Claim:}  $(X,\omega^\prime)$ is minimal as a symplectic $G$-manifold for sufficiently small $\epsilon$. \\

Suppose there is a disjoint union of $\omega^\prime$-symplectic $(-1)$-spheres 
$\{C_i\}$, such that for any $g\in G$, $g\cdot C_i=C_j$ for some $j$. Note that for sufficiently
small $\epsilon$, we have $\E_\w=\E_{\w'}$ because $K_{\omega}=K_{\omega^\prime}$.
Let $e_i$ be the class of $C_i$.
It follows easily that for each $i$, $e_i\in\E_\w$. Now pick a $G$-invariant $J$ compatible with $\w$.
Since $(X,\omega)$ is monotone, each $e_i$ is represented by an embedded $J$-holomorphic $(-1)$-sphere $\hat{C}_i$.
Notice that $\{e_i\}$ has the following properties:
 $e_i\cdot e_j=0$ for $i\neq j$, and for any $g\in G$, $g\cdot e_i=e_j$ for some $j$. It follows that
$\{\hat{C}_i\}$ is a disjoint union of $\omega$-symplectic $(-1)$-spheres which is invariant under the $G$-action. This
contradicts the minimality assumption on $(X,\omega)$, hence the claim. 

We apply the argument for case i) to $(X,\omega^\prime)$. 
Consequently there is a reduced basis $\{H, E_1,\cdots,E_N\}$
(w.r.t. $\omega^\prime$), and a $G$-invariant fibration $\pi^\prime:X\rightarrow \s^2$, whose fiber class is $H-E_1$ and
the singular fibers are pairs of $(-1)$-spheres representing $E_j,H-E_1-E_j$ for $j>1$. Note that by taking $\epsilon$ small,
we have $E_i\in\E_{\w'}=\E_\w$.
Finally, observe the following
crucial property $(\sum_{g\in G} g\cdot E_j)^2=0$ for any $j>1$ because $E_k$ and $H-E_1-E_k$ must appear in pairs in the sum.

Now let $J$ be any $G$-invariant $\omega$-compatible almost complex structure.  
For each fixed $j>1$, 
since $(X,\omega)$ is monotone and $E_j\in\E_\w$, $E_j$ is represented 
by an embedded $J$-holomorphic $(-1)$-sphere $C$. Let $\Lambda:=\cup_{g\in G} \;g\cdot C$. 
Then the set $\Lambda$ is a union of finitely many distinct $J$-holomorphic
$(-1)$-spheres. Since $(\sum_{g\in G} g\cdot E_j)^2=0$ for any $j>1$, it follows easily that 
$\Lambda^2=0$.

Let $\{\Lambda_i\}$ be the set of connected components of $\Lambda$. Since $G$ acts on the connected components of $\Lambda$ transitively, 
it follows that for any $i\neq k$, $\Lambda_i^2=\Lambda_k^2$. Clearly $\Lambda^2=\sum_i \Lambda_i^2$, which implies
that $\Lambda_i^2=0$ for any $i$. Now $\Lambda_i^2=0$, together with 
the fact that $\Lambda_i$ is a union of finitely many distinct $J$-holomorphic $(-1)$-spheres and $\Lambda_i$ is connected,
implies that $\Lambda_i$ consists of two $(-1)$-spheres intersecting transversely at a single point. 
We claim that the pair of $(-1)$-spheres in each $\Lambda_i$ have classes $E_k,H-E_1-E_k$ for some $k>1$. 
This is because for each $g\in G$, $g\cdot E_j$ is either $E_k$ or $H-E_1-E_k$ for some $k>1$, and for each $k>1$,
there is a $g\in G$ such that $g\cdot E_k=H-E_1-E_k$. 
By the same argument as in case i), we obtain a $G$-invariant fibration on $X$ as desired, independent of the choice of $j$. 

The statement that $N\ge 5$ was proved in Lemma \ref{l:selfint}. The statement that $N\ne 6$ follows from Lemma 
\ref{l:nosixblowup} below. 
This completes the proof of Theorem \ref{t:X-structure}.



\begin{lemma}\label{l:nosixblowup}

Let $\pi:X\rightarrow \s^2$ be a minimal symplectic $G$-conic bundle where 
$X=\CP^2\#6\overline\CP^2$. Then for any $G$-invariant, compatible almost complex
structure $J$, there is a $G$-invariant $J$-holomorphic $(-1)$-sphere. 
\end{lemma}

\begin{proof}
Let $H,E_1,\cdots,E_6$ be a basis of $H^2(X)$ with standard interaction matrix such that
$H-E_1$ is the fiber class and $E_2,\cdots, E_6$ are $(-1)$-spheres contained in singular
fibers. Note that the canonical class $K_X=-3H+E_1+\cdots +E_6$.

With this understood, note that $C=-K_X-(H-E_1)=2H-\sum_{6\ge i\ge2}E_i\in H^2(X)^G$,
and $C$ is an exceptional class. Hence for any given $G$-invariant, compatible almost complex
structure $J$, $C$ has a $J$-holomorphic representative, which admits a decomposition of irreducible components
\begin{equation}\label{e:FE}
     C=\hat F+\hat E,
\end{equation}
where $\hat F$ is the sum of components contained in the fibers of the $G$-conic bundle (the \textit{vertical class}), 
and $\hat E$ is the sum of other components (the \textit{horizontal class}).  

By Lemma \ref{4-3}, $E_1$ must have a $J$-holomorphic representative: otherwise, it has a stable 
curve representative where one of the component has to be a section since $E_1\cdot F=1$.  
Such a section has self-intersection less than $-1$, a contradiction.

To continue with our proof, we first show that $\hat E$ does not contain $E_1$-components.  Suppose the multiplicity of the $E_1$-component is $k$, since all irreducible components pairs with $H-E_1$ non-negatively, $(C-kE_1)\cdot (H-E_1)\ge0$, hence $k\le2$.  

Furthermore, if the multiplicity is $2$ (meaning either there is a doubly covered component or two components), $\hat E=2E_1$.  Otherwise, $\hat E\cdot (H-E_1)\ge3>C\cdot(H-E_1)$, violating the positivity of intersection with $H-E_1$.

For the multiplicity $2$ case, $\hat F=2H-2E_1-\sum_{6\ge i\ge2} E_i$, while all possible irreducible components are of the form $H-E_1-E_i$, $E_i$ or $H-E_1$ for $6\ge i\ge2$.  A simple check shows this cannot be consistent with \eqref{e:FE}.

For the multiplicity $1$ case, $\hat E=E_1+E'$. $E'$ has $E'\cdot(H-E_1)=1$ hence a section.  Therefore, $E'\ge-1$ from Lemma \ref{4-3}.  Since $E'\ne E_1$, its coefficient of $H$ under the reduced basis must be positive.  One may again easily check $\hat F$ cannot be represented as a combination of classes of forms $H-E_1-E_i$, $E_i$ or $H-E_1$.

Now since the equivariant $J$-holomorphic representative of $C$ does not contain $E_1$-components, it is disjoint from the $E_1$-section.  Therefore, $\hat F\cdot E_1=0$, which implies the sum of vertical components $\hat F=\sum m_sE_s$ for some $s\ne 1$.  However, there is always an element $g\in G$ which sends $E_s$ to $H-E_1-E_s$ for any $s\ne 1$.  This forces $\hat F=0$, and $C=\hat E$.

At last, again note that $\hat E$ has at most $2$ components by positivity of intersection with $H-E_1$.  We assume $\hat E$ has two components, which are both sections $S_1$ and $S_2$.  

$$C^2=(S_1+S_2)^2=-1=S_1^2+S_2^2+2S_1S_2\ge -2+2S_1\cdot S_2$$
from Lemma \ref{4-3}.  Therefore, $2S_1\cdot S_2\le 1$, which implies $S_1\cdot S_2=0$.  This implies these are disjoint sections $S_1^2=-1$ and $S_2^2=0$.  A simple check on class $C$ shows there are no such decomposition of section classes (recall that $[S_1]\ne E_1$).

Summarizing, we have showed that the $J$-representative of $C$ is indeed irreducible, 
which is a $G$-invariant $J$-holomorphic $(-1)$-sphere. Hence the lemma.

\end{proof}

To compare with known results in algebraic geometry, it seems worth to record the following easy consequence. 

\begin{theorem}\label{t:sixPoints}
Let $(X,\omega)$ be a minimal symplectic rational $G$-surface where 
$X=\C\P^2\# 6\overline{\C\P^2}$. Then the invariant lattice $H^2(X)^G$ must be of rank
$1$. In particular, a minimal complex rational $G$-surface $X$ which is $\C\P^2$ blown up at $6$ points
must be Del Pezzo with $\text{Pic}(X)^G=\Z$.
\end{theorem}



\begin{proof}

Under the above assumptions, Lemma \ref{l:nosixblowup} implies $(X,\w)$ is not a $G$-conic bundle, hence $rank(H^2(X)^G)=1$ by Theorem \ref{t:X-structure}.

For the second statement, notice that a conic bundle on a minimal complex rational $G$-surface defines a minimal symplectic $G$-conic
bundle with respect to any $G$-invariant K\"{a}hler form.
Theorem 3.8 of \cite{DI} asserts then $X$ must be a del Pezzo surface if $rankH^2(X)^G=1$.

\end{proof}



\section{The structure of $G$}
\subsection{Proof of Theorem \ref{t:G-structure}, $4\le N\le 8$} 
\label{sub:proof_of_theorem_ref}

We start with the following observation.

\begin{lemma}\label{l:faithful}
Suppose $(X,\omega)$ is monotone. If $N\geq 4$, then the representation of $G$ on $H^2(X)$ is faithful.
 \end{lemma}
 
 \begin{proof}
 Fix a $G$-invariant $J$.
 Let $g\in G$ be any element acting trivially on $H^2(X)$. Then $g$ fixes every element 
 $E\in \E_\w$, which implies that
 all the $J$-holomorphic $(-1)$-spheres are invariant under $g$. Now let $C_1$ be the $J$-holomorphic 
 $(-1)$-sphere representing $E_1$, and for each $1<j\leq N$, let $C_j$ be the $J$-holomorphic $(-1)$-sphere representing
 $H-E_1-E_j$. Then it is clear that $C_1$ intersects each $C_j$, $j>1$, transversely at one point and the $C_j$'s are mutually
 disjoint. It follows that the cardinality of the set $C_1\cap (\cup_j C_j)$ is $N-1$. On the other hand, each point in
 $C_1\cap (\cup_j C_j)$ is fixed under $g$, hence the action of $g$ on $C_1$ contains at least $N-1$ fixed points. When
 $N\geq 4$, it follows that $C_1$ must be fixed by $g$. A similar argument shows that every $J$-holomorphic $(-1)$-sphere
 is fixed by $g$. Since there are $J$-holomorphic $(-1)$-spheres intersecting transversely at a point, the action of $g$ on the
 tangent space at the intersection point must be trivial, which shows that $g$ must be trivial. Hence the lemma.
 
 \end{proof}

Assume $(X,\omega)$ is a symplectic rational $G$-surface with $\text{rank} H^2(X)^G=1$.

Since $K_\omega\in H^2(X)$ is fixed under the action of $G$, there is an induced representation of
$G$ on the orthogonal complement $R_N$, which is faithful by Lemma \ref{l:faithful}. This gives rise to a monomorphism 
$\rho: G\rightarrow W_N$. On the other hand, $H^2(X)^G$ is spanned by $K_\omega$, so that $R_N^G=\{0\}$. This implies that
$$
\frac{1}{|G|}\sum_{g\in G} trace\{ \rho(g):R_N\rightarrow R_N\}=\text{rank }R_N^G=0.
$$
The proof for Case 1 is completed.

\vspace{5mm}

\subsection{Proof of Theorem \ref{t:G-structure}, $N=3$} 
\label{sub:ProofGstructure}

Our main objective in this case is to show that $G$ contains an index $2$
 subgroup which is isomorphic to an imprimitive finite subgroup of $PGL(3)$.
 
 We first describe the list of subgroups of $PGL(3)$ involved. For simplicity, we will adapt the
 following convention from \cite{DI}: we denote an element  $T\in PGL(3)$ by the image of
 $[z_0,z_1,z_2]\in\C\P^2$ under $T$. 
 
 Here is the list of imprimitive finite subgroups of $PGL(3)$ up to conjugacy (see Theorem 4.7
 in \cite{DI}), where $\mu_r=\exp(2\pi i/r)$ is the $r$-th root of unity.
 
 \begin{itemize}
 \item $G_n$, generated by the following elements of $PGL(3)$:
 $$
 [\mu_n z_0,z_1,z_2],\;\; [z_0,\mu_nz_1,z_2],\;\; [z_2,z_0,z_1]
 $$
 The group $G_n$ is isomorphic to a semi-direct product of $(\Z_n)^2$ and $\Z_3$.
 \item $\tilde{G}_n$, generated by the following elements of $PGL(3)$:
 $$
  [\mu_n z_0,z_1,z_2],\;\; [z_0,\mu_n z_1,z_2],\;\; [z_0,z_2, z_1],\;\; [z_2,z_0,z_1]
 $$
The group $\tilde{G}_n$ is isomorphic to a semi-direct product of $(\Z_n)^2$ and $S_3$.
\item $G_{n,k,s}$, where $k>1$, $k|n$, and $s^2-s+1=0 \pmod{k}$. It is generated by 
the following elements of $PGL(3)$:
 $$
 [\mu_{n/k} z_0,z_1,z_2],\;\; [\mu_n^sz_0,\mu_nz_1,z_2],\;\; [z_2,z_0,z_1]
 $$
 The group $G_{n,k,s}$ is isomorphic to a semi-direct product of $\Z_n\times\Z_{n/k}$ and $\Z_3$.
 \item $\tilde{G}_{n,3,2}$, generated by the following elements of $PGL(3)$:
 $$
 [\mu_{n/3} z_0,z_1,z_2],\;\; [\mu_n^2 z_0,\mu_n z_1,z_2],\;\; [z_0,z_2,z_1],\;\; [z_1,z_0,z_2]
 $$
The group $\tilde{G}_{n,3,2}$ is isomorphic to a semi-direct product of $\Z_n\times\Z_{n/3}$ 
and $S_3$.
 \end{itemize}
 
 \vspace{2mm}

\subsubsection{Preliminaries: factorization of $G$ by exceptional spheres}
Let $\{H,E_1,E_2,E_3\}$ be a reduced basis. Then
$$
\E_\w=\{E_1,E_2,E_3, H-E_1-E_2,H-E_1-E_3,H-E_2-E_3\}.
$$ 
Since $(X,\omega)$ is monotone, the classes in $\E_\w$ have the same area, and consequently, each class in $\E_\w$ is represented 
by a $J$-holomorphic $(-1)$-sphere for any fixed $J$, which we assume to be $G$-invariant. 
Let $\Lambda$ be the union of 
these six $(-1)$-spheres. The intersection pattern of these curves can be described by a hexagon, where each edge represents 
a $(-1)$-sphere and each vertex represents an intersection point
(See Figure 1). For simplicity, for each $E\in\E_\w$ we shall use the same notation to denote the
corresponding $(-1)$-sphere.

\begin{figure}[tb]
  \centering
  \includegraphics[scale=1.2]{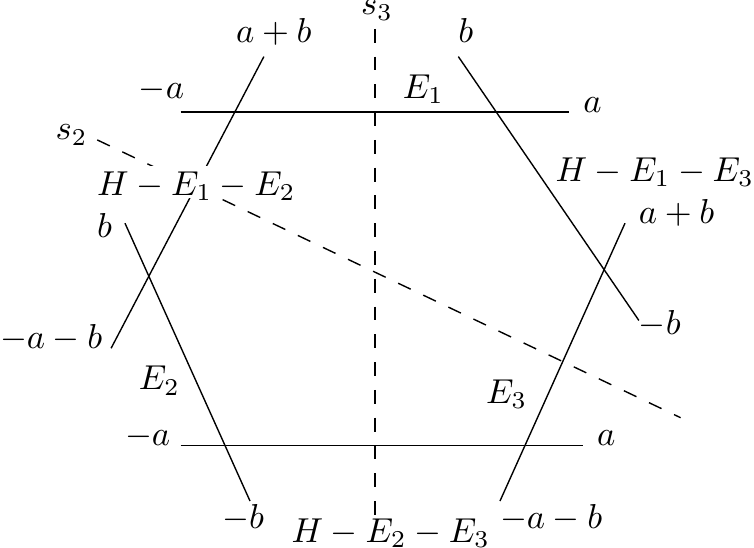}
  \caption{Exceptional configurations and action weights in $\CP^2\#3\overline{\CP^2}$}
  \label{fig:figure1}
\end{figure}

Obviously there is an induced $G$-action on $\Lambda$.

\begin{lemma}\label{l:Kalternative}
The action of $G$ on the components of $\Lambda$ is transitive, and consequently, there is a short exact sequence 
\begin{equation}\label{e:SESG}
1\rightarrow \Gamma\rightarrow G\rightarrow K\rightarrow 1,     
\end{equation}

where $\Gamma$ is isomorphic to a subgroup of $\s^1\times\s^1$, and $K$ is either $D_{12}$,
the full automorphism group of the hexagon, or the cyclic subgroup of order $6$.
\end{lemma}

\begin{proof}
Let $C$ be any of the $(-1)$-spheres. Then the class of the union $\cup_{g\in G} g\cdot C$ must lie in
$H^2(X)^G$. Since $H^2(X)^G=\text{Span }(K_\omega)$, the class of $\cup_{g\in G} g\cdot C$ must be a
multiple of $K_\omega$. On the other hand, the class of $\Lambda$ equals $-K_\w$, from which it follows 
easily that $\Lambda=\cup_{g\in G} g\cdot C$. This proves that the action of $G$ on components 
of $\Lambda$ is transitive.

The action of $G$ on $\Lambda$ gives rise to a short exact sequence 
$$
1\rightarrow \Gamma\rightarrow G\rightarrow K\rightarrow 1,
$$
 where $\Gamma$ is the normal subgroup of $G$ consisting of elements which leave each $(-1)$-sphere
 invariant. The quotient group $K=G/\Gamma$ has an effective, transitive action on the hexagon $\Lambda$,
 which must be either the automorphism group $D_{12}$ of the hexagon, or the cyclic subgroup of order $6$.

 To see that $\Gamma$ is a subgroup of $\s^1\times\s^1$, we look at
 the action of $\Gamma$ on the tangent space of any intersection point of two adjacent $(-1)$-spheres in $\Lambda$.
 The action preserves a pair of complex lines intersecting transversely, giving a natural isomorphism of $\Gamma$ to
 a subgroup of $\s^1\times\s^1$.
 
 \end{proof}
 
In fact, we may identify $K$ geometrically as follows. There is a natural isomorphism $D_{12}=\Z_2\times S_3$, where the latter is the Weyl group $W_N$ 
 of the corresponding root lattice $R_N$, which is generated by $H-E_1-E_2-E_3$, $E_1-E_2$ and $E_2-E_3$.
 In this sense, $\Z_2=\langle s_1\rangle$ where $s_1$ is the rotation of $180$ degrees of the
 hexagon, and $S_3=\langle s_2,s_3\rangle$ where $s_2$, $s_3$ are the reflections of the 
 hexagon which switches $E_1$ and $E_2$, $E_2$ and $E_3$ respectively. Note that $s_2s_3$,
 which is $s_3$ followed by $s_2$, is a counter-clockwise rotation of $120$ degrees of the hexagon. 
 It follows that the short exact sequence $1\rightarrow \Gamma\rightarrow G\rightarrow K\rightarrow 1$
 is the same as the one obtained from the induced action of $G$ on the root lattice $R_N$, 
 with $\Gamma$ being the subgroup of $G$ acting trivially on $H^2(X)$. Under this identification, one of the following is true by Lemma \ref{l:Kalternative}
\begin{itemize}
  \item $K$ is 
 generated by $s_1$, $s_2$ and $s_3$ if it is the cyclic subgroup of order $12$.
  \item $K$ is 
 generated by $s_1$ and $s_2s_3$ if it is the cyclic subgroup of order $6$.
\end{itemize}
 
 In order to understand the structure of $G$, we begin by getting more information about
 the subgroup $\Gamma$. To this end, we fix a monomorphism 
 $\Gamma\rightarrow \s^1\times\s^1$ induced from the action of $\Gamma$ on the tangent
 space of the intersection point of $E_1$ and $H-E_1-E_3$, where the first $\s^1$-factor
 is from the action on $E_1$. We summarize the main objects in consideration as follows for readers' convenience:

\begin{itemize}
  \item $\Gamma\tl G$ is the subgroup with trivial homological action,
  \item $K=G/\Gamma$,   
  \item $\rho_i: \Gamma\rightarrow \s^1$, $i=1,2$, are the projection to the two $\s^1$-factors,
  \item $\Gamma_i=\ker \rho_i$,
  \item $\Gamma_i^\prime=\text{image }\rho_i\subset \s^1$. 
  \item $\tilde{\Gamma}_i$ is a subgroup of $\Gamma$ such that 
 $\rho_i: \tilde{\Gamma}_i\rightarrow \Gamma_i^\prime$ is isomorphic. 
\end{itemize}

\begin{lemma}\label{l:Gamma}\hfill
\begin{enumerate}
  \item $\Gamma_i$, $\Gamma_i'$ are cyclic,
  \item ord $(\Gamma_1)$=ord $(\Gamma_2)$ and ord $(\Gamma_1')$=ord $(\Gamma_2')$
  \item ord $(\Gamma_i)|$ord $(\Gamma_i')$
\end{enumerate}

\noindent We will hence denote $n:=ord(\Gamma_i')$ and $k:=$ord$(\Gamma_i')$/ord$(\Gamma_i)$
\end{lemma}

\begin{proof}
It is clear that $\Gamma_i^\prime$ is cyclic. $\Gamma_i$ is also cyclic because both 
$\rho_2|_{\Gamma_1}$ and $\rho_1|_{\Gamma_2}$ are injective. 
 Note that this also shows that the order
of $\Gamma_1^\prime$ (resp. $\Gamma_2^\prime$) is divisible by the order of 
$\Gamma_2$ (resp. $\Gamma_1$). 

Finally, $\Gamma_1$ and $\Gamma_2$ have the same order. This is because if we
let $g\in G$ be an element whose action on the hexagon is a counter-clockwise 
rotation of $60$ degrees, then $g\Gamma_2g^{-1}=\Gamma_1$. Consequently, the order of 
$\Gamma_i^\prime=\Gamma/\Gamma_i$ is independent of $i$.

\end{proof}

 \begin{lemma}\label{l:GammaTilde}
      The subgroup $\tilde{\Gamma}_i<\Gamma$ exists.  In particular, since $\Gamma$ is Abelian, 
      $\Gamma\cong\Gamma_i\times \tilde{\Gamma}_i$ for $i=1,2$.
  \end{lemma}
    
 \begin{proof}
 Let $h\in \Gamma$ be an element 
 such that $\rho_1(h)$ is a generator of $\Gamma_1^\prime$. Since $\rho_1(h)^n=1$, 
 we have $h^n\in \Gamma_1$. We claim $h^n=1$. This is because $\rho_2|_{\Gamma_1}$ 
 is injective, so that if $h^n\neq 1$ in $\Gamma_1$, then $\rho_2(h)^n=\rho_2(h^n)\neq 1$ in 
 $\Gamma_2^\prime$. But this contradicts the fact that the order of $\Gamma_2^\prime$ 
 equals $n$. Hence $h^n=1$. With this understood, we simply take 
 $\tilde{\Gamma}_1$ to be the subgroup generated by $h$.
 
 \end{proof}
 
\subsubsection{Rotation numbers}\label{3-2-2} Next we recall some basic facts about rotation numbers. Let $h\in \Gamma$ be any element and 
 $E\in\E$ be any $(-1)$-sphere which is invariant under $h$. Then either $h$ fixes $E$ or 
 $h$ acts on $E$
 nontrivially. In the latter case, $h$ fixes two points on $E$.  The following fact is a straightforward local computation:
\vskip5mm
 \noindent \textbf{Fact:} \textit{if we let $(a,b)$
  be the rotation numbers of $h$ at one of the fixed point, where $a$ is the tangential weight and
  $b$ is the weight in the normal direction, then the rotation numbers at the other fixed point are
  $(-a, b+a)$, with the second number being the weight in the normal direction.} 

\vskip 5mm

 Explicitly, the tangential weight being $a$ means that the action of $h$ is given by multiplication of 
 $\exp(2a\pi\sqrt{-1}/ord(h))$ and similarly for the normal weight.  Here we orient 
 the $(-1)$-sphere by the almost complex structure $J$ and orient the normal direction accordingly,
 so there is no sign ambiguity on $a$, $b$. Note that even when $E$ is fixed by $h$, this 
 continues to make sense, with the understanding that $a=0$ in this case.

 With the preceding understood, we order the exceptional curves and their intersections according 
 to the counter-clockwise orientation of the hexagon in Figure \ref{fig:figure1}, and ask $E_1$ to be the first exceptional curve.

 For example,  we say $H-E_1-E_3$ is \textbf{before} $E_1$ and $H-E_1-E_2$ is \textbf{after} $E_1$, 
 and the intersection point of $H-E_1-E_3$ and $E_1$ is the \textbf{first} fixed point on $E_1$, etc.

 Now for any $h\in \Gamma$, we denote by $(a,b)$ the rotation
 numbers at the intersection point of $H-E_1-E_3$ and $E_1$, with $a$ being the weight in the
 direction tangent to $E_1$. According to the orientation of the hexagon, this is the first fixed
 point of $h$ on $E_1$. With this understood, the rotation numbers at the  fixed points of $h$
 on $(-1)$-spheres are given below according to the orientation of the hexagon,
 $$
 (a+b,-a), (b,-a-b), (-a,-b), (-a-b,a), (-b,a+b).
 $$
 Finally, we remark that $h$ is completely determined by the rotation numbers at the six vertices of the hexagon.
 
 Throughout the rest of the proof of Theorem \ref{t:G-structure}, 

\begin{center}
\textit{$g\in G$ will denote those elements which act on the hexagon by a \\
 counter-clockwise rotation of $60$ degrees, }
  
\end{center}

 and we shall investigate the action of $g$ on 
 $\Gamma=\Gamma_1\times \tilde{\Gamma}_1$ given by conjugation, i.e., $h\mapsto ghg^{-1}$, 
 $\forall h\in \Gamma$.

 Here is the first corollary of the analysis on rotation numbers: note that the action of $g^3$, 
 which is a rotation of
 $180$ degrees, sends every pair of rotation numbers to its negative, i.e., $(a,b)\mapsto (-a,-b)$,
 $(a+b,-a)\mapsto (-a-b,a)$ and $(b,-a-b)\mapsto (-b,a+b)$. This implies that 
 $$
 g^3hg^{-3}=h^{-1}, \;\;\forall h\in \Gamma.
 $$
 Since $g^3$ is sent to $s_1\in K$ under the homomorphism $G\rightarrow K$ in \eqref{e:SESG}, 
 we obtain the following lemma:

\begin{lemma}\label{l:basic}
 (1) For any element of $G$ which is sent to $s_1\in K$ under $G\rightarrow K$, its
 action on $\Gamma$ by conjugation sends $h\in \Gamma$ to $h^{-1}$. (2) $g^6=1$.
 \end{lemma}
 
 \begin{proof}
 It remains to show that $g^6=1$.
 First, note that $g^6\in \Gamma$. Secondly, the action of $g^3$ on $g^6$ by conjugation 
 is trivial, but the action of $g^3$ on $\Gamma$ sends 
 $h$ to $h^{-1}$ as we just showed, from which we see that either $g^6=1$ or $g^6$ 
 is an involution.  We rule out the latter by showing that the action of $g$ on any 
 involution of $\Gamma$ is nontrivial, and $G$ is a semi-direct product of $\Gamma$ 
 and $K$ follows. 
 
Let $\tau\in \Gamma$ be any involution. Without loss of generality, we assume $\tau$ acts 
nontrivially on $E_1$. Then in the rotation numbers of the two fixed points, $(a,b)$ and 
$(-a,b+a)$, $a=1$ must be true because $\tau$ has order $2$, and furthermore, either $b=0$ or $b=1$. 
If $b=0$, the rotation numbers for the action of $\tau$ at the six vertices are 
 $$
 (1,0), (1,-1), (0,-1), (-1,0), (-1,1), (0,1).
 $$
 If $b=1$, the rotation numbers for the action of $\tau$ at the six vertices are 
 $$
 (1,1), (0,-1), (1,0), (-1,-1), (0,1), (-1,0).
 $$
In the former case, the rotation numbers for the action of $g\tau g^{-1}$ are
$$
 (0,1), (1,0), (1,-1), (0,-1), (-1,0), (-1,1),
 $$
which shows $g\tau g^{-1}\neq \tau$. In the latter case, the rotation numbers for the action 
of $g\tau g^{-1}$ are
$$
(-1,0), (1,1), (0,-1), (1,0), (-1,-1), (0,1), 
 $$
which also shows $g\tau g^{-1}\neq \tau$. This finishes the proof of the lemma.

\end{proof}

With these preparation, we may now continue our proof according to the alternative by 
Lemma \ref{l:Kalternative} 

 \begin{itemize}
   \item Case A: $K=\Z_6$, and
   \item Case B: $K=D_{12}$.
 \end{itemize}

 \subsubsection{Case A: $K=\Z_6$} We begin with the following observation which follows
 immediately from the fact that $g^6=1$.
 
 \begin{proposition}
$G$ is a semi-direct product of $\Gamma$ and $K$.
 \end{proposition}
 
 Let $\langle \Gamma,g^2\rangle$ denote the index $2$ subgroup of $G$ 
generated by $\Gamma$ and $g^2$.

\begin{proposition}\label{p:3-4}
If $K=\Z_6$, then $\langle \Gamma,g^2\rangle$ is isomorphic to $G_n$ when $k=1$, and is
isomorphic to $G_{n,k,s}$ when $k>1$, where $n=|\Gamma_1^\prime|=|\tilde{\Gamma}_1|$ 
and $n/k=|\Gamma_1|$.
\end{proposition}

\begin{proof}
 Let $h_1\in \Gamma_1$ be the generator whose rotation numbers at the six vertices of 
 the hexagon are
 $$
 (0,1), (1,0), (1,-1), (0,-1), (-1,0), (-1,1),
 $$
and let $\tilde{h}_1\in \tilde{\Gamma}_1$ be the generator whose rotation numbers 
at the six vertices of the hexagon are
 $$
 (1,b), (1+b,-1), (b,-1-b), (-1,-b), (-1-b,1), (-b,1+b).
 $$
This is can be shown by examining the weight on the first intersection (the one between 
$H-E_1-E_3$ and $E_1$), because on the base it should project to that of $\Gamma_1'$. Then $g^2h_1g^{-2}$ has rotation numbers
$$ 
(-1,0), (-1,1), (0,1), (1,0), (1,-1), (0,-1), 
$$
and $g^2\tilde{h}_1 g^{-2}$ has rotation numbers
$$
(-1-b,1), (-b,1+b), (1,b), (1+b,-1), (b,-1-b), (-1,-b). 
$$
Let  $g^2h_1g^{-2}=\tilde{h}_1^{kl} h_1^u$ for some $l$ and $u$ (note that $k\cdot ord(\Gamma_i)=ord(\Gamma'_i)$). 
Then comparing the rotation
numbers we have 
$$
(-1,0)= (l,bl)+(0,u),
$$
which implies that $l=-1$ and $u=b$. Similarly, let $g^2\tilde{h}_1 g^{-2}=\tilde{h}_1^{m} h_1^v$
for some $m$ and $v$, we have 
$$
(-1-b,1)=(m,mb)+(0,kv),
$$
which implies that $b^2+b+1=kv \pmod{n}$ and $m=-1-b$. Putting together, we have
$$
g^2h_1g^{-2}=\tilde{h}_1^{-k} h_1^b,\;\;\;\; g^2\tilde{h}_1 g^{-2}=\tilde{h}_1^{-b-1} h_1^v.
$$
We will need a different presentation of $\langle \Gamma,g^2\rangle$. To this end, 
we let $t_1=h_1$, $t_2=\tilde{h}_1^{-1}$, and moreover, we set $s=-b$. Then we have
$$
 g^2t_1g^{-2}=t_2^{k} t_1^{-s},\;\;\;\; g^2t_2 g^{-2}=t_2^{s-1} t_1^{-v},
 $$
 where $s^2-s+1=kv \pmod{n}$. With this presentation, one can identify the subgroup
 $\langle \Gamma,g^2\rangle$ with an imprimitive finite subgroup of $PGL(3)$.
 More precisely, when $k=1$, i.e., $\tilde{\Gamma}_1$ and $\Gamma_1$ have the same 
 order, we can actually take $\tilde{\Gamma}_1$ to be $\Gamma_2$, which corresponds to 
 $b=0$. Then $s=-b=0$, and with $k=1$, we have $v=1$. In this case. $\langle \Gamma,g^2\rangle$ 
 is isomorphic to $G_n$ by identifying $t_1= [\mu_{n} z_0,z_1,z_2]$, 
 $t_2=[z_0,\mu_n z_1,z_2]$, and $g^2=[z_2,z_0,z_1]$.
 When $k>1$, $\langle \Gamma,g^2\rangle$ is isomorphic to the group
 $G_{n,k,s}$, by identifying $t_1= [\mu_{n/k} z_0,z_1,z_2]$, 
 $t_2=[\mu_n^s z_0,\mu_nz_1,z_2]$, and $g^2=[z_2,z_0,z_1]$.
 This finishes off the proof of Proposition \ref{p:3-4}.
 
 \end{proof}
 
 \subsubsection{Case B: $K=D_{12}$} We will first show that $G$ is a semi-direct product of $\Gamma$ 
 and $K$. 
 
 \begin{lemma}\label{l:s3}
 There exists an involution in $G$ which is sent to the reflection $s_3$ under $G\rightarrow K$.
 Moreover, 
 \begin{itemize}
 \item for any $h\in\Gamma$ and any such involution $\tau$, $h\tau h^{-1}=\hat{h}\tau$, where
 $h,\hat{h}$ are related as follows: if the rotation numbers of $h$ at the first fixed point on $E_1$
 (according to the orientation of the hexagon which is counter-clockwise) are $(a,b)$, then the 
 rotation numbers of $\hat{h}$ at the first fixed point on $E_1$ are $(2a,-a)$.
 \item for any two such involutions $\tau,\tau^\prime$, $\tau^\prime=h\tau$ for some $h\in\Gamma$
 whose rotation numbers at the first fixed point on $E_1$ are $(2a,-a)$ for some $a$.
 \end{itemize}
 \end{lemma}
 
 We note that analogous statements hold for the reflections $s_2$ and $s_2s_3s_2$ in
 $K$.
 
 \begin{proof}
 Suppose $\tau\in G$ is sent to $s_3$ under $G\rightarrow K$, then $\tau$ leaves the 
 $(-1)$-spheres $E_1$ and $H-E_2-E_3$ invariant. Let $h_1\in \Gamma_1$ be the generator 
 with rotation numbers $(0,1)$ at the first fixed point. By examining the rotation numbers,
 one can check easily that $\tau h_1\tau^{-1}=h_1$. On the other hand, it is easily seen that
 $\tau^2\in \Gamma$ fixes $4$ points (two fixed points from $\tau$ and two from intersections with other exceptional curves) hence the whole $(-1)$-sphere $E_1$.  Therefore, $\tau^2\in \Gamma_1$.

  The
 key observation is that $\tau^2=h_1^{2b}$ is an even power of $h_1$. To see this, note that 
 $\tau$ has two fixed points on $E_1$ and their rotation numbers are
 $(a,b)$ and $(-a,b+a)$ for some $a,b$. Furthermore, the order of $\tau$ must be even, say $2m$.
Since $\tau^2$ fixes the $(-1)$-sphere $E_1$, we must have $2a=2m$, and the rotation numbers
of $\tau^2$ at the two fixed-points are $(0, 2b)$ and $(0,2b)$. Comparing with the rotation
numbers with $h_1$, we see easily that $\tau^2=h_1^{2b}$. With $\tau h_1\tau^{-1}=h_1$, we 
see easily that $\tau h_1^{-b}$ is an involution which is sent to $s_3$ under $G\rightarrow K$.

Now we consider any involution $\tau$ which is sent to $s_3$. For any $h\in \Gamma$, let 
$x,y$ be the first and second fixed points on $E_1$, and suppose the rotation numbers of $h$
at $x$ are $(a,b)$. Then $\tau(x)=y$ and $\tau: T_x E_1
\rightarrow T_y E_1$. It is easily seen that $h\tau h^{-1}: T_x E_1\rightarrow T_y E_1$ equals
$(-a) \tau (-a)=-2a \tau$ (here $(-a)$ means multiplication by $\exp(2(-a)\pi\sqrt{-1}/ord(h))$).
Similarly, $h\tau h^{-1}= (a+b)\tau (-b)=a\tau$ in the normal direction.
Hence $h\tau h^{-1}=\hat{h}\tau$ for some $\hat{h}$ whose rotation numbers at $y$ are
$(-2a,a)$. It follows easily that the rotation numbers of $\hat{h}$ at $x$ are $(2a,-a)$.

Finally, let $\tau,\tau^\prime$ be any two involutions sent to $s_3$. We let 
\[
  f=d\tau:T_xX\to T_yX
\]
\[
  g=d\tau':T_yX\to T_xX.
\]
 Then 

\[
   g\circ f=d(\tau^\prime \tau): T_x X\rightarrow T_x X
 \]
 \[
   g^{-1}\circ f^{-1}=d(\tau \tau'): T_y X\rightarrow T_y X
 \]
 
   Note that
$\tau^\prime\tau=h\in\Gamma$. Since $g^{-1}\circ f^{-1}=(f\circ (g\circ f)\circ f^{-1})^{-1}$,
we see easily that the rotation numbers of $h$ at $y$ are the negative of the rotation numbers
at $x$. If the rotation numbers at $x$ are $(c,d)$, then the rotation numbers at $y$ are
$(-c,c+d)$ (the second number in each pair stands for the weight in the normal direction).
This gives rise to the relation $c+d=-d$, so that the rotation numbers of $h$ at $x$ are $(-2d,d)$.
Setting $d=-a$ we proved that $\tau^\prime=h\tau$ for some $h\in\Gamma$ whose rotation 
numbers at the first fixed point on $E_1$ are $(2a,-a)$ for some $a$, as we claimed.

\end{proof}

\begin{proposition}
The group $G$ is a semi-direct product of $\Gamma$ and $K$.
\end{proposition}

\begin{proof}
Recall that we fixed an element $g\in G$ which is sent to a counter-clockwise rotation of $60$ degrees under $G\rightarrow K$.
We set $\tau_1=g^3$, which is an involution by Lemma \ref{l:basic}. We pick another involution $\tau_3\in G$ which is sent to $s_3\in K$
from Lemma \ref{l:s3}. We shall first show that we can always arrange to have $\tau_1\tau_3\tau_1=\tau_3$.

By Lemma \ref{l:basic}, $\tau_1\tau_3\tau_1=h^\prime\tau_3$ for some $h^\prime\in \Gamma$ whose rotation numbers at the first fixed point on $E_1$
are $(2a_3,-a_3)$ for some $a_3$. On the other hand, if we replace $\tau_3$ by $h\tau_3h^{-1}=\hat{h}\tau_3$, then from Lemma \ref{l:basic}
$$
\tau_1(\hat{h}\tau_3)\tau_1=\hat{h}^{-1}\tau_1\tau_3\tau_1=\hat{h}^{-2}h^\prime (\hat{h}\tau_3).
$$
Hence if there exists an $\hat{h}\in \Gamma$ such that $\hat{h}^2=h^\prime$, which is equivalent to $a_3$ being even, then we can
replace $\tau_3$ by $h\tau_3h^{-1}$ for an $h\in\Gamma$ to achieve the commutativity property. 

To show that $a_3$ is even, we pick an involution $\tau_2$ which is sent to $s_2\in K$ (by an 
analog of Lemma \ref{l:basic}) and consider $\tau_1\tau_2\tau_1$. By the corresponding version of
Lemma \ref{l:basic}, we see that $\tau_1\tau_2\tau_1=\tilde{h}\tau_2$ for some $\tilde{h}\in \Gamma$ whose rotation numbers at the
second fixed point on $E_1$ are $(2a_2,-a_2)$ for some $a_2$. It follows easily that the rotation numbers of $\tilde{h}$ at the
first fixed point on $E_1$ are $(a_2,a_2)$. 

Now $\tau_2\tau_3$ is sent to a counter-clockwise rotation of $120$ degrees under 
$G\rightarrow K$, so that there exists an
$h\in \Gamma$ such that $\tau_2\tau_3=hg^2$. Now 
$$
\tau_1(\tau_2\tau_3)\tau_1=\tilde{h}\tau_2h^\prime\tau_3=\tilde{h}h^\prime k\tau_2\tau_3
$$
for some $k\in \Gamma$ whose rotation numbers can be determined as follows: since the rotation 
numbers of $h^\prime$ at the second fixed point on $E_1$ are $(a_3,-2a_3)$, by the analog of
Lemma \ref{l:basic} the rotation numbers of $k$ at the second fixed point on $E_1$ are $(-2a_3,a_3)$.
It follows that the rotation numbers of $k$ at the first fixed point are $(-a_3,-a_3)$.
With this understood, note that the rotation numbers of $\tilde{h}h^\prime k$ at the first fixed point on $E_1$ are $(a_2+a_3,a_2-2a_3)$. On the other hand,
\begin{equation}\label{e:tau23}
     \tau_1(\tau_2\tau_3)\tau_1=\tau_1(hg^2)\tau_1=h^{-1} \tau_1g^2\tau_1=h^{-1} g^2=h^{-2}\tau_2\tau_3,
\end{equation}

which implies that both $a_2+a_3$, $a_2-2a_3$ are even. It follows that $a_3$ is even, and hence there is an involution
$\tau_3$ sent to $s_3\in K$ with the property that $\tau_1\tau_3\tau_1=\tau_3$.

We set $\tau_2:=g\tau_3g^{-1}$. Then $\tau_2$ is an involution sent to $s_2\in K$ which naturally satisfies $\tau_1\tau_2\tau_1=\tau_2$.
We will show that $\tau_2,\tau_3$ satisfy the relation $\tau_2\tau_3\tau_2=\tau_3\tau_2\tau_3$. 
Note that with this relation, the subgroup generated by
$\tau_2,\tau_3$ is isomorphic to $S_3$. Together with the involution $\tau_1$, we obtain a lifting of $K=\Z_2\times S_3$ in $G$, proving that $G$ is
a semi-direct product of $\Gamma$ and $K$. 

As we have seen earlier, $\tau_2\tau_3=hg^2$ for some $h\in \Gamma$, where $h^2=1$ from \eqref{e:tau23} because both $\tau_2,\tau_3$ commute
with $\tau_1$. The rotation numbers of $h$ at the first fixed point on $E_1$ must be one of the following: (i) $(0,1)$, (ii) $(1,0)$,
(iii) $(1,1)$. We claim that in case (i), we have $\tau_2\tau_3\tau_2=\tau_3\tau_2\tau_3$. To see this,
$$
\tau_2\tau_3\tau_2=hg^2\tau_2=h\tau_3 g^2=\tau_3hg^2=\tau_3\tau_2\tau_3,
$$
where we use the fact $h\tau_3=\tau_3h$ because the rotation numbers of $h$ are $(0,1)$. 

It remains to rule out the cases (ii) and (iii). To this end, we compute
$$
g\tau_2\tau_3g^{-1}=g^2\tau_3g^{-2} \tau_2=(h\tau_2\tau_3)\tau_3(h\tau_2\tau_3)^{-1}\tau_2
=h\tau_2\tau_3\tau_2 h\tau_2=hhk\tau_2\tau_3\tau_2\tau_2=k\tau_2\tau_3,
$$
where the rotation numbers of $k\in\Gamma$ can be determined as follows. Note that 
$\tau_2\tau_3\tau_2$ is sent to the reflection $s_2s_3s_2\in K$, so the rotation numbers of $k$
at the first fixed point on $H-E_1-E_3$ can be determined by an analog of Lemma \ref{l:basic}. In case (ii), 
the rotation numbers of $h$ at the first fixed point on $H-E_1-E_3$ are $(0,1)$, so by an analog 
of Lemma \ref{l:basic}, the rotation numbers of $k$ at the first fixed point on $H-E_1-E_3$ are $(0,0)$,
i.e., $k$ is trivial in this case. In case (iii), the rotation numbers of $h$ at the first fixed point on 
$H-E_1-E_3$ are $(1,0)$, so by an analog of Lemma \ref{l:basic}, the rotation numbers of $k$ at the first fixed point on 
$H-E_1-E_3$ are $(0,1)$. It follows that the rotation numbers of $k$ at the first
fixed point on $E_1$ are $(1,0)$ in this case.

On the other hand, $g(hg^2)g^{-1}=ghg^{-1}g^2=h^\prime g^2$ for some $h^\prime \in\Gamma$, 
where the rotation numbers of $h^\prime$ at the first fixed point on $E_1$ are $(0,1)$ in case (ii) 
and $(1,0)$ in case (iii). Since in both cases, $h^\prime h\neq k$, so we reached a contradiction.
This ruled out the cases (ii) and (iii), and the proposition is proved.

\end{proof}

Finally, we show that $G$ contains an index $2$ subgroup which is isomorphic to an imprimitive
subgroup of $PGL(3)$. To this end, we fix a lifting $K^\prime$ of $K$ to $G$, and let $g\in K^\prime$
be an element of order $6$ and $\tau\in K^\prime$ be the involution sent to $s_3\in K$. We denote
by $\langle \Gamma,g^2,\tau \rangle$ the subgroup generated by the elements in $\Gamma$,
$g^2$ and $\tau$.

\begin{proposition}
Suppose $K=D_{12}$. Then $\langle \Gamma,g^2,\tau \rangle$ is isomorphic to the imprimitive
finite subgroup $\tilde{G}_n$ of $PGL(3)$ when $k=1$, and is isomorphic to $\tilde{G}_{n,3,2}$
when $k>1$, where $n=|\Gamma_1^\prime|=|\tilde{\Gamma}_1|$ and $n/k=|\Gamma_1|$.
\end{proposition}

\begin{proof}
 Let $h_1\in \Gamma_1$ be the generator whose rotation numbers at the six vertices of 
 the hexagon are
 $$
 (0,1), (1,0), (1,-1), (0,-1), (-1,0), (-1,1),
 $$
and let $\tilde{h}_1\in \tilde{\Gamma}_1$ be the generator whose rotation numbers 
at the six vertices of the hexagon are
 $$
 (1,b), (1+b,-1), (b,-1-b), (-1,-b), (-1-b,1), (-b,1+b).
 $$
Then the rotation numbers of $\tau \tilde{h}_1\tau$ are 
$$
(-1,1+b), (b,1), (1+b,-b), (1,-1-b), (-b,-1), (-1-b,b).
$$
Writing $\tau \tilde{h}_1\tau =\tilde{h}_1^l h_1^u$ for some $l,u$, we get
$$
(-1, 1+b)=(l,lb)+(0,ku),
$$
which implies that $l=-1$ and $2b+1=ku \pmod{n}$.  

(i) Assume $\tilde{\Gamma}_1$ and $\Gamma_1$ have the same order $n$, i.e. $k=1$. 
Recall from the proof of Proposition \ref{p:3-4} that in this case, 
$b=0$, so that $u=1$ and $\tau \tilde{h}_1\tau=\tilde{h}_1^{-1} h_1$. Renaming $t_1=h_1$, $t_2=\tilde{h}_1^{-1}$
as in the proof of Proposition \ref{p:3-4}, we get
$$
 g^2t_1g^{-2}=t_2,\;\;\;\; g^2t_2 g^{-2}=t_2^{-1} t_1^{-1}, \;\;\;\; \tau t_1\tau =t_1,
 \;\;\;\; \tau t_2\tau =t_2^{-1} t_1^{-1}.
$$
With this presentation, the subgroup $\langle \Gamma,g^2,\tau \rangle$ can be identified with 
$\tilde{G}_n$ by 
$$
g^2=[z_2,z_0,z_1],\;\; \tau=[z_0,z_2,z_1],\;\; t_1=[\mu_n z_0,z_1,z_2],\;\; t_2=[z_0,\mu_n z_1,z_2].
$$
 
 (ii) Assume $\tilde{\Gamma}_1$ and $\Gamma_1$ have different orders, i.e., 
  $|\tilde{\Gamma}_1|=n$, $|\Gamma_1|=n/k$, with $k>1$. In this case, we first note that $2b+1=0\pmod{k}$. 
  On the other hand, recall from the proof of Proposition \ref{p:3-4} that $b^2+b+1=0 \pmod{k}$. It follows that  $b=1\pmod{k}$ and
$k=3$. With this understood, note that one can modify $\tilde{h}_1$ by a suitable power of $h_1$ to arrange so that $b=-2$. With this choice,
we then have $s=-b=2$ and $v=1$, where $s,v$ appear in the relations (see the proof of Proposition \ref{p:3-4})
$$
 g^2t_1g^{-2}=t_2^kt_1^{-s},\;\;\;\; g^2t_2 g^{-2}=t_2^{s-1} t_1^{-v}.
$$
  Moreover, $b=-2$ implies $u=-1$, hence the presentation of $\langle \Gamma,g^2,\tau \rangle$: 
 $$
 g^2t_1g^{-2}=t_2^{3} t_1^{-2},\;\;\;\; g^2t_2 g^{-2}=t_2 t_1^{-1},  \;\;\;\; \tau t_1\tau =t_1,
 \;\;\;\; \tau t_2\tau=t_2^{-1} t_1
$$
With this presentation, the subgroup $\langle \Gamma,g^2,\tau \rangle$ can be identified
with $\tilde{G}_{n,3,2}$ by identifying 
$$
g^2=[z_2,z_0,z_1],\;\; \tau=[z_0,z_2,z_1],\;\; t_1=[\mu_{n/3} z_0,z_1,z_2],\;\; 
t_2=[\mu_n^2 z_0,\mu_n z_1,z_2].
$$
\end{proof}

It is clear that $G$ is a semi-direct product of the imprimitive subgroup of $PGL(3)$ and $\Z_2$. 
The proof of Theorem \ref{t:X-structure} is completed.

\vspace{5mm}
 
The rest of this section is occupied by the proof of Theorem \ref{t:Q-structure}, where we assume $(X,\omega)$ 
admits a minimal symplectic $G$-conic bundle structure $\pi:X\rightarrow \s^2$. Furthermore,
we assume $N\geq 4$. 
Recall that $Q$ is the subgroup of $G$ which leaves each fiber of 
$\pi$ invariant, and $G_0$ is the subgroup of $Q$ which acts trivially on $H^2(X)$.

\vspace{2mm}

\subsection{Proof of Theorem \ref{t:Q-structure}} 
\label{subsec:proof_of_theorem_ref_}


\vspace{2mm}

We begin with some useful observations about the rotation numbers of an element of $Q$
at a fixed point. 

\vskip 3mm
\noindent\textit{Rotation numbers of fixed points.}
\vskip 3mm

Suppose $q\in X$ is fixed by a nontrivial element $g\in Q$ and $q$ is not 
the singular point of a singular fiber. Then $g$ must fix the symplectic orthogonal direction of the fiber at $q$, 
because $g$ induces a trivial action on the base. It follows that the rotation numbers of $g$ 
at $q$ are $(a,0)$ for some $a\neq 0$. 

Furthermore, if $q$ lies in a singular fiber, then $g$ must leave the $(-1)$-sphere containing $q$ invariant, 
and the rotation numbers of $g$ at the other 
fixed point, which is the singular point of the singular fiber, must be $(a,-a)$. (See Section \ref{3-2-2})

\vskip 3mm
\noindent\textit{Structure of fixed point sets.}
\vskip 3mm

Note that the fixed-point set of a nontrivial element $g\in Q$ consists of embedded $J$-holomorphic
curves and isolated points. Since each regular fiber contains two fixed points of $g$, it follows that the 
fixed-point set of $g$ consists of a bisection and isolated points which are the singular points
of those singular fibers of which $g$ leaves each component invariant. 

Clearly, one of the following must be true:

\begin{enumerate}[(i)]
  \item $g$ leaves components of a singular fiber invariant.  

  Then there is a fixed point $q$ as described above, and the singularity $p$ on the fiber has rotation number $(a,-a)$ hence is an isolated fixed point (and there is one more fixed point on either component of the fiber).
  \item $g$ switches the 2 components of the singular fiber.  

  Then

  \begin{enumerate}
    \item the singularity $p$ on the fiber is contained in the fixed bisection,
    \item $p$ is a branched point of the double branched covering from the bisection
to the base,
    \item $g$ must be an involution, because if $g^2\neq 1$, then $p$
must be an isolated fixed point of $g^2$ from (i), a contradiction.
  \end{enumerate}

\end{enumerate}


With the preceding understood, note that the subgroup $G_0$ can be identified with the subgroup
which leaves each $(-1)$-sphere in a singular fiber invariant. Since there are $N-1\geq 3$
singular fibers, the induced action of $G_0$ on the base $\s^2$ has at least $3$ fixed points.
It follows that the action of $G_0$ on the base must be trivial, and $G_0$ is a subgroup of $Q$.
It is clear that either $G_0$ is trivial, or it is finite cyclic. 

Furthermore, the fixed-point set of 
each nontrivial element of $G_0$ consists of $N-1$ isolated points with rotation numbers 
$(a,-a)$ for some $a\neq 0$ (these are the singular points of the singular fibers) and two disjoint fixed sections from the two other fixed points other than the singular points on each fiber. 
The two fixed sections must have the same self-intersection number because there exists a $g\in G$ which switches 
the two sections.
Applying Lemma \ref{l:selfint}  to the fixed sections, we see that each has self-intersection $(N-1)/2$; in particular, $N$ must be odd
if $G_0$ is nontrivial. 
Finally, we note that each element of $Q\setminus G_0$ must be an involution (the square fixes 4 points in a general fiber: intersections with the bisection, and the intersections of the two disjoint sections as above).

\begin{lemma}
The group $Q$ contains an involution; in particular, it has an even order. 
\end{lemma}

\begin{proof}
Let $F$ be the singular fiber which contains the exceptional sphere representing $E_N$. The minimality assumption implies that there is a $g\in G$ which 
switches the two $(-1)$-spheres in $F$. Clearly, $g$ has an even order, say $2m$. If $g\in Q$, then as we have shown
earlier, $g$ must be an involution, and we are done. Suppose $m>1$. We let $h=g^m$, which is an involution. We claim
that $h\in Q$. 

Suppose $h$ is not contained in $Q$. Then $h$ induces a rotation on the base, so that the fixed-point set of $h$ must be 
contained in the two fibers which $h$ leaves invariant. With this understood, we claim that $h$ must fix one of the $(-1)$-spheres
in $F$. Denote $\Sigma$ as the two dimensional component of the fixed point set of $h$, which could be empty
If $h$ does not fix either of the $(-1)$-spheres in $F$, we must have $E_N\cdot \Sigma=0$. On the other hand, 
by Proposition 5.1 in \cite{Ed}, we have
$$
(h\cdot E_N)\cdot E_N= E_N\cdot \Sigma \pmod{2}.
$$
This is a contradiction because $h\cdot E_N=E_N$ or $H-E_1-E_N$, so that $(h\cdot E_N)\cdot E_N=\pm 1$. Hence $h$ fixes
one of the $(-1)$-spheres in $F$. Now since $g$ commutes with $h$ and $g$ switches the two $(-1)$-spheres in $F$, $h$ must
also fix the other $(-1)$-sphere. But this clearly contradicts the fact that $h$ is nontrivial. Hence $h\in Q$
and the lemma is proved.

\end{proof}

\begin{proposition}
If $G_0$ is nontrivial, say, of finite order $m>1$, and $Q\neq G_0$, then $Q$ is the dihedral group $D_{2m}$.
Moreover, any involution $\tau\in Q\setminus G_0$ switches the two fixed sections of $G_0$, hence the two $(-1)$-spheres in each singular fiber.
\end{proposition}

\begin{proof}
For any $g\in Q$, since $G_0$ is normal in $Q$, $g$ leaves the two fixed sections of $G_0$
invariant. Note that if $g$ leaves each of the section invariant, then it must fix both of them because the
induced action of $g$ on the base is trivial. Consequently, $Q/G_0$ is either trivial or $\Z_2$,
depending on whether there is a $g\in Q$ switching the two fixed sections of $G_0$. It follows easily that 
every element in $Q\setminus G_0$ switches the two $(-1)$-spheres in each singular fiber.
Finally, if $Q\neq G_0$, then $Q$ must be the corresponding dihedral group because each element
in $Q\setminus G_0$ is an involution. ($g^2$ preserves all exceptional spheres hence in $G_0$) This finishes the proof of the proposition.

\end{proof}

Next we consider the case where $G_0$ is trivial. Let $\Sigma$ be the set of singular fibers. 

\begin{proposition}
Suppose $G_0$ is trivial. Then $Q=\Z_2$ or $(\Z_2)^2$. In the latter case, let $\tau_1,\tau_2,\tau_3$ be the
distinct involutions in $Q$. Then $\Sigma$ is partitioned into subsets $\Sigma_1$, $\Sigma_2$,
$\Sigma_3$, where $\Sigma_i$ parametrizes the set of singular fibers of which $\tau_i$
leaves each $(-1)$-sphere invariant, and $\# \Sigma_i\equiv N-1 \pmod{2}$,  for $i=1,2,3$.
\end{proposition}

\begin{proof}
First of all, $Q$ consists of involutions as $G_0$ is trivial. Suppose $Q\neq \Z_2$, and let $\tau, \tau^\prime\in Q$
be two distinct nontrivial elements. We claim that there is no singular fiber such that both $\tau,\tau^\prime$ leave both of
the $(-1)$-spheres in this fiber invariant. This is because if there is such a singular fiber, then by examining the
action of $\tau\tau^\prime$ at the singular point of the fiber, we see $\tau\tau^\prime$ must be trivial (the rotation numbers at this signular point is
$(0,0)$). But this contradicts the assumption that $\tau\neq \tau^\prime$. 

With the preceding understood, let $\tau_1,\cdots,\tau_n$, $n>1$, be the distinct involutions in $Q$, and 
let $\Sigma_i$ be the set of singular fibers of which $\tau_i$ leaves each $(-1)$-sphere invariant. Then the
previous paragraph shows that $\Sigma_i\cap\Sigma_j=\emptyset $ for $i\neq j$. On the other hand, suppose
$\tau_k=\tau_i\tau_j$, then $\Sigma\setminus (\Sigma_i\cup\Sigma_j)\subset \Sigma_k$, implying 
$$
\Sigma\setminus (\Sigma_i\cup\Sigma_j)=\Sigma_k. 
$$
It follows easily that $n=3$ and $Q=(\Z_2)^2$. 

It remains to see that for each $i$, $\# \Sigma_i=N-1 \pmod{2}$.
Consider the fixed-point set $S_i$ of $\tau_i$. Then $S_i$ is a bisection and the projection of $S_i$ onto the base
is a double branched covering which ramifies exactly at the singular points of those singular fibers not parametrized by
the set $\Sigma_i$. Since the number of ramifications must be even, we have $\# \Sigma_i=N-1 \pmod{2}$ as claimed.
 
 \end{proof}
 
The proof of Theorem \ref{t:Q-structure} is completed.

\section{Minimality and equivariant symplectic cones} 
\label{sec:equivariantcone}

Now let $X=\C\P^2\# N\overline{\C\P^2}$, $N\geq 2$,
which is equipped with a smooth action of a finite group $G$. Suppose there is a $G$-invariant
symplectic form $\omega_0$ on $X$ such that the corresponding symplectic rational 
$G$-surface $(X,\omega_0)$ is minimal. With this understood, we denote by 
$\Omega(X,G)$ the set of $G$-invariant symplectic forms on $X$.

The following is a crucial observation.

\begin{lemma}\label{4-2}
For any $\omega\in \Omega(X,G)$, the canonical class $K_\omega=K_{\omega_0}$ or $-K_{\omega_0}$.
\end{lemma}

\begin{proof}

The lemma is obvious if $H^2(X)^G$ has rank $1$ because $[\w_0],[\w]$ and $K_\w$ are proportional. 

Assume that  $H^2(X)^G$ has rank $2$. Recall from the proof of Theorem \ref{t:X-structure}, there is a reduced basis 
$H,E_1,\cdots, E_N$ of $(X,\omega_0)$ such that $K_{\w_0}= -3H + E_1 +E_2 + \cdots + 
E_N$. Now after changing $\omega$ by sign and with a further scaling if necessary, we can write 
$[\w]=-K_{\w_0}+bF= (3+b)H - (1+b) E_1 -\cdots -E_N$ for some $b\in\R$. 
We claim that $[\w]$ is always a \textit{reduced class} in the sense of \cite{LL}.  
To see this, note that the condition $[\w]^2>0$ implies $4b>N-9$, which implies $b\ge-1$
since $N\geq 5$. This gives $3+b>1+b>0$ and $3+b=1+b+1+1$.  Now the conclusion 
follows from \cite[Lemma 3.4 (part 5)]{LL}, which asserts that any symplectic class with reduced form has canonical class 
$K_{\w_0}$.

\end{proof}

Recall that we have shown that a minimal symplectic rational $G$-surface where $X=\C\P^2\# N\overline{\C\P^2}$
admits a minimal symplectic $G$-conic bundle only when  $N\geq 5$ and $N\neq 6$.
The following lemma deals with precisely these cases and gives the converse of this fact. 

\begin{lemma}\label{l:4-1}
Let $(X,\omega)$ be a symplectic rational $G$-surface where $X=\C\P^2\# N\overline{\C\P^2}$
with $N\geq 5$ and $N\neq 6$. Suppose $(X,\omega)$ admits a minimal symplectic $G$-conic bundle 
structure. Then for any $\omega^\prime \in \Omega(X,G)$ such that $K_{\omega^\prime}=K_\omega$,
the symplectic rational $G$-surface $(X,\omega^\prime)$
is minimal. 
\end{lemma}

\begin{proof}
By Lemma \ref{l:ConicRank}, $H^2(X)^G$ is of rank $2$ spanned by $K_\w$ and the fiber class $F$ of
the symplectic $G$-conic bundle. Now suppose to the contrary that $(X,\omega^\prime)$
is not minimal. Then there is a $G$-invariant, disjoint union of $\omega^\prime$-symplectic
$(-1)$-spheres $C_1,C_2,\cdots, C_m$. Let $C=C_1+\cdots +C_m$. Then since 
$K_{\omega^\prime}=K_\omega$, we have 
$$
-m=C^2=K_\omega\cdot C.
$$
On the other hand, $C\in H^2(X)^G$, so that 
\begin{equation}\label{e:4-1}
     C=aK_\omega+bF,\hskip 2mm a,b\in \Z.
\end{equation}


The key ingredient for deriving a contradiction is the fact that $F\cdot C\geq 0$,  which 
as a corollary implies that $a<0$, and so that $F\cdot C>0$. 
To see this, Note that $F$ is represented by an embedded
$J$-holomorphic sphere $S$ where $J$ is $\omega$-compatible. 
On the other hand, since $K_{\omega^\prime}=K_\omega$, $\omega$ and $\omega^\prime$ have the same set of
symplectic $(-1)$-classes. Hence since the class of each $C_i$ is represented by a 
$\omega^\prime$-symplectic $(-1)$-sphere, it is also represented by a $\omega$-symplectic 
$(-1)$-sphere. Consequently, the class of $C_i$ can be represented by $\cup_j m_j D_j$ 
where $D_j$
is a $J$-holomorphic curve and $m_j>0$. Now since $S$ is irreducible and $S^2=0$, we
have $S\cdot D_j\geq 0$ for each $j$, which implies that $F\cdot C_i\geq 0$ for each $i$.
Hence our claim that $F\cdot C\geq 0$. 

With this understood, we consider the pairing of $C=aK_\omega+bF$ with $K_\omega$ 
and then square both sides, we have 

$$aK_\omega^2-2b=-m$$
$$-m=a^2K_\omega^2-4ab$$
which gives 

\begin{equation}\label{e:coprime}
m=-\frac{a^2K_\omega^2}{2a-1}.    
\end{equation}
Notice that $m>0$, hence $K_\omega^2>0$, which excludes all cases for $N\ge9$.  
Moreover, $a^2$ and $2a-1$ are co-prime for all $a\le -1$, therefore, $2a-1|K_\omega^2$.  But this is not
possible because of the assumption $N\geq 5$ and $N\neq 6$.

\end{proof}


{\bf Proof of Theorem \ref{t:minimality}}

\vspace{3mm}

The claim regarding the canonical class is proved in Lemma \ref{4-2}. The minimality claims
are trivial if $H^2(X)^G$ is of rank $1$. For the case $\text{rank} H^2(X)^G>1$, they
follow from Lemma \ref{4-2}, Lemma \ref{l:4-1} for $N\ge5$ and $N\neq 6$; and Lemma \ref{l:nosixblowup} for $N=6$. 
Theorem \ref{t:X-structure} implies these are the only cases to consider.  Combining these results with
Theorem 3.8 of \cite{DI}, we cover the case (2) of minimal complex rational $G$-surfaces.\qed

\vspace{3mm}

In Theorem \ref{t:sixPoints}, the case of complex rational $G$-surface admits an alternative proof, which we sketch
it here. Let $X$ be a minimal complex rational $G$-surface which is $\C\P^2$ blown up at $6$ points,
and assume $X$ is a conic bundle. 
Then by Proposition 5.2 of \cite{DI}, $X$ must be Del Pezzo, and hence has a $G$-invariant
monotone K\"{a}hler form $\omega$.  Part (1) of Lemma \ref{4-5} below implies there are two distinct fiber classes in $H^2(X)^G$, but $N=6$ contradicts part (2) of the lemma.

\vskip 5mm
We now turn to the proof of Theorem \ref{t:sympCone}.  Fixing the canonical class $K_{\omega_0}$, we recall that $F\in H^2(X)^G$ is called a \textbf{fiber class} if it is the class of the fibers of a symplectic
$G$-conic bundle on $X$ for some $G$-invariant symplectic form $\w$ with $K_\w=K_{\omega_0}$.

 \begin{lemma}\label{4-5}
 Suppose $H^2(X)^G$ has rank $2$.
 
\begin{enumerate}[(1)]
   \item If $X$ admits a $G$-invariant monotone symplectic form, then there are at least two distinct
 fiber classes in $H^2(X)^G$.
   \item Suppose $F,F^\prime\in H^2(X)^G$ are distinct fiber classes. Then $F+F'=-aK_{\omega_0}$ 
 for some integer $a>0$, and $N=5,7$ or $8$ where $a=1, 2, 4$ respectively. In particular, there
 are at most two distinct fiber classes in $H^2(X)^G$.

 \end{enumerate}
\end{lemma}
 
\begin{proof}
To prove (1), let $\omega$ be a $G$-invariant monotone form on $X$, and let 
$F$ be the fiber class of the $G$-conic bundle structure obtained from Theorem 1.3. Pick
a $G$-invariant closed $2$-form $\eta$ representing $F$. Then for sufficiently small 
$\epsilon\neq 0$, the $G$-invariant symplectic form $\omega^\prime:=\omega+\epsilon\eta$ 
is non-monotone, and the symplectic $G$-manifold $(X,\omega^\prime)$ is minimal
(as shown in the proof of Theorem \ref{t:X-structure}). 
By Theorem \ref{t:X-structure}, $(X,\omega^\prime)$ admits a symplectic $G$-conic bundle with small
fiber. An easy check with the symplectic areas shows that for $\epsilon>0$, the small fiber class (whose area is twice the minimal exceptional spheres)
of the symplectic $G$-conic bundle equals $F$, but for $\epsilon<0$, it is not $F$.


To prove (2), let $\omega,\omega^\prime$ be the $G$-invariant symplectic forms associated with the
symplectic $G$-conic bundles whose fiber classes are $F,F^\prime$ respectively. 
For simplicity, we set $K=K_{\omega_0}$. 

Write 
\begin{equation}\label{e:KF}
     F^\prime=-aK +b F,
\end{equation}
for some $a,b\in\Z$. Note that $K\cdot F=K\cdot F^\prime=-2$, 
$F^2=(F^\prime)^2=0$, and the assumption that
$F\neq F^\prime$ implies $a\neq0$.  Then by pairing \eqref{e:KF} with $K$, $F$ and $F'$, respectively, one has $-aK^2=-4$ and $b=-1$. Therefore, $a$ and $K^2$
are both divisors of $4$, and $F+F'=-aK$.

We claim $2a=F\cdot F'\geq 0$, which implies $a>0$ and $N=5,7,8$. The point is that $F$
can be represented by an embedded $J$-holomorphic sphere $V$ with $V^2=0$, where 
$J$ is $\omega$-compatible. Since $K_\w=K_{\w^\prime}$, this fact implies the Gromov-Witten invariant of $F'$ is nontrivial, 
hence $F^\prime$
can be represented by a stable $J$-holomorphic curves, from which our claim $F\cdot F'\geq 0$
follows easily.

\end{proof}

\begin{lemma}
Suppose $G_0$ is nontrivial. Then there is a unique fiber class.
\end{lemma}

\begin{proof}
We note first that the $2$-dimensional fixed point set of $G_0$ consists of two embedded
$2$-spheres $S_1,S_2$ , each with self-intersection $-(N-1)/2$ (in particular, $N$ must be odd).
Suppose to the contrary that there are two distinct fiber classes $F,F^\prime$. Then $S_1$
(and $S_2$) is a $J$-holomorphic section of the corresponding symplectic $G$-conic bundles
with fiber classes $F,F^\prime$, for an appropriate $\omega$ or $\omega^\prime$-compatible
$G$-invariant $J$. This implies that $S_1\cdot F=S_1\cdot F^\prime=1$. On the other hand,
by Lemma \ref{4-5}, $F+F^\prime=-aK_{\omega_0}$ for some $a>0$, implying that
$K_{\omega_0}\cdot S_1<0$. This violates the adjunction formula because $S_1^2=
-(N-1)/2\leq -2$, and $K_\w=K_{\w^\prime}=K_{\omega_0}$.

\end{proof}

{\bf Proof of Theorem \ref{t:sympCone}}

\vspace{3mm}

Part (1) and part (2) follow immediately from Lemma \ref{4-5} and Lemma 4.4, and the proof of Theorem \ref{t:X-structure}.  In the case (2) when there is a unique fiber class $F$ for the $G$-conic bundle, any $G$-invariant symplectic form has the form $(3H-E_1-\cdots-E_m)+bF$.  From Theorem \ref{t:X-structure}, $\w(E_1)\ge\w(E_k)$, hence $b\ge0$.  This also shows $C(X,G,F)\cup C(X,G,F')=C(X,G)$.

To see that $C(X,G,F)\cap C(X,G,F^\prime)$ is either empty or consists of
classes of $G$-invariant monotone symplectic forms, let $[\omega]\in
C(X,G,F)$ and $[\omega^\prime]\in C(X,G,F^\prime)$ such that $[\omega]=[\omega^\prime]$.
Then $[\omega]=-a K_{\omega_0}+bF$, 
$[\omega^\prime]=-a^\prime K_{\omega_0}+b^\prime F^\prime$, where $a,a^\prime>0$ and
$b,b^\prime\geq 0$. If both $b,b^\prime$ are non-zero, then $[\omega]=[\omega^\prime]$,
together with the fact that $F+F^\prime$ is a multiple of $K_{\omega_0}$ (cf. Lemma \ref{4-5}),
would imply that $F,F^\prime$ are linearly dependent, a contradiction.

To see (3), first we note that if $[\omega]\in \hat{C}(X,G,F)$, then for any 
$\delta>\delta_{\omega,F}$,
there is an $\omega^\prime\in\Omega(X,G)$ such that $[\omega^\prime]\in \hat{C}(X,G,F)$
with $\delta_{\omega^\prime,F}=\delta$. We can simply take $\omega^\prime:=
\omega+ (\delta-\delta_{\omega,F})\pi^\ast \eta$, where $\pi$ is a symplectic $G$-conic 
bundle on $(X,\omega)$ with fiber class $F$, and $\eta$ is an area form on the base of $\pi$
with total area $1$. Secondly, if $\delta_{X,G,F}>0$, then it can not be attained. Suppose to
the contrary that there is an $\omega$ such that $\delta_{\omega,F}=\delta_{X,G,F}$.
Then take $0<\epsilon<\delta_{X,G,F}$ sufficiently small, the $G$-invariant form
$\omega^\prime:=\omega-\epsilon \eta$, where $\eta$ is a $G$-invariant closed $2$-form representing $F$, is a symplectic form. The condition $\epsilon<\delta_{X,G,F}$ implies
that $[\omega^\prime]\in \hat{C}(X,G,F)$, which contradicts the definition of $\delta_{X,G,F}$. 

\vspace{3mm}

We end this section with a uniqueness result on the subgroup $Q$ from Definition 1.7.
 
 \begin{proposition}
 Suppose $(X,G)$ has a unique fiber class $F$. Then the normal subgroup $Q$ of $G$ is uniquely
 determined, i.e., it is independent of $\omega\in \Omega(X,G)$, nor the symplectic $G$-conic
 bundle structure on $(X,\omega)$ involved in the definition of $Q$. 
 \end{proposition}
 
 \begin{proof}
 Let $\omega,\omega^\prime\in \Omega(X,G)$, and let $\pi, \pi^\prime$ be symplectic 
 $G$-conic bundles with fiber class $F$, and let $Q,Q^\prime$ be the subgroups of $G$
 defined using $\pi,\pi^\prime$ respectively. Let $H,E_1,\cdots,E_N$ and 
 $H^\prime,E_1^\prime,\cdots,E_N^\prime$ be a reduced basis associated to $\pi,\pi^\prime$
 respectively.
 
  Let $Q_{(H,E_i)}=\{g\in G| g\cdot E_j=E_j \mbox{ or } H-E_1-E_j\}$. We claim $Q=Q_{(H,E_i)}$.
 First, it is clear that $Q\subset Q_{(H,E_i)}$. Secondly, if $g\in Q_{(H,E_i)}$, then $g$ leaves
 each singular fiber invariant. Since the number of singular fibers is $N-1$ which is greater
 than $3$, the induced action of $g$ on the base $\s^2$ has at least $3$ fixed points. This
 implies that the action of $g$ on the base must be trivial, and $g\in Q$. Hence $Q=Q_{(H,E_i)}$.
 Similarly, $Q^\prime=Q_{(H^\prime,E_i^\prime)}$.
 
 If we normalize so that $\omega(F)=\omega^\prime(F)=2$, then for each $j=2,\cdots,N$,
 $\omega(E_j^\prime)=\omega^\prime(E_j)=1$ also. In particular,  
 $H^\prime,E_1^\prime,\cdots,E_N^\prime$ is a reduced basis for $(X,\omega)$. By
 Lemma \ref{l:structureLemma}, and the fact that $\omega$ is not monotone, we have,
 for each $j>1$, $E_j^\prime=E_k$
 or $H-E_1-E_k$ for some $k>1$. It follows that $H^\prime-E_1^\prime-E_j^\prime=
 H-E_1-E_k$ or $E_k$ for the same $k$. From these relations we see immediately that
 $Q_{(H,E_i)}=Q_{(H^\prime,E_i^\prime)}$. Hence $Q=Q^\prime$.
  
 \end{proof}

\end{document}